\documentclass[11pt,a4paper]{article}
\usepackage{amssymb}
\usepackage{amsmath}
\usepackage{latexsym}
\usepackage{amsthm}
\usepackage{dsfont}
\usepackage{mathtools}
\usepackage{cellspace}
\usepackage{tikz}
\usepackage{stmaryrd}

\newtheorem{thm}{Theorem}
\newtheorem{lem}[thm]{Lemma}
\newtheorem{cor}[thm]{Corollary}
\theoremstyle{definition}

\usepackage{xpatch}
\xpatchcmd{\proof}{\itshape}{\normalfont\proofnameformat}{}{}
\newcommand{\proofnameformat}{}

\begin{document}

\renewcommand{\proofnameformat}{\bfseries}

\begin{center}
{\Large\textbf{Random walks and quadratic number fields}}

\vspace{10mm}

\textbf{Bence Borda}

{\footnotesize

University of Sussex

Brighton, BN1 9RH, United Kingdom

Email: \texttt{b.borda@sussex.ac.uk}}

\vspace{5mm}

{\footnotesize \textbf{Keywords:} Wasserstein metric, $L^2$ discrepancy, quadratic irrational, zeta function, Diophantine approximation}

{\footnotesize \textbf{Mathematics Subject Classification (2020):} 60G50, 11K38, 11E45, 11R42}
\end{center}

\vspace{5mm}

\begin{abstract}
We establish a novel type of connection between random walks and analytic number theory. Working with a random walk on the circle group $\mathbb{R}/\mathbb{Z}$ in which each step is a random integer multiple of a given quadratic irrational $\alpha$, we show that the rate of convergence to uniformity in the quadratic Wasserstein metric (also known as the periodic $L^2$ discrepancy) is governed by deep arithmetic invariants of the ring of algebraic integers of the real quadratic field $\mathbb{Q}(\alpha)$, such as fundamental units and special values of zeta functions.
\end{abstract}

\section{Introduction}

The behavior of a random walk on the circle group is governed by a delicate interplay between classical probability, Fourier analysis and Diophantine approximation. The main goal of this paper is to establish a further, rather surprising connection to analytic number theory. As a prototypical example, we work with a random walk in which each step is a random integer multiple of a given quadratic irrational $\alpha$. We show that the rate of convergence of such a random walk to uniformity is inextricably linked to the arithmetic of the real quadratic field $\mathbb{Q}(\alpha)$. In particular, we demonstrate that the random walk is able to detect deep arithmetic invariants such as fundamental units in the ring of algebraic integers, and special values of zeta functions.

To state our main result, let $X_1, X_2, \ldots$ be independent, identically distributed (i.i.d.) integer-valued random variables, and set $S_n = \sum_{k=1}^n X_k$. Given $\alpha \in \mathbb{R}$, the sequence $S_n \alpha \pmod{\mathbb{Z}}$ is then a random walk on the circle group $\mathbb{T}=\mathbb{R}/\mathbb{Z}$. The empirical measure $\mu_N = N^{-1} \sum_{n=1}^N \delta_{S_n \alpha \pmod{\mathbb{Z}}}$ converges to the normalized Haar measure $\lambda$ on $\mathbb{T}$ under very general assumptions. Our main result is the precise asymptotics of the expected value and the variance of the distance squared from the empirical measure to uniformity in the quadratic Wasserstein metric $W_{\mathbb{T},2}$ on $\mathbb{T}$. See Section \ref{contextsection} for definitions, notation and further context.
\begin{thm}\label{maintheorem} Assume $\mathbb{E} (X_1) =0$ and $0<\mathbb{E} (X_1^2) < \infty$, and $\alpha$ is a quadratic irrational. Then
\[ \mathbb{E} \left( W_{\mathbb{T},2}^2 (\mu_N, \lambda) \right) \sim \frac{L^2 c_1(L \alpha)}{\sigma^2} \cdot \frac{\log N}{N} \quad \textrm{and} \quad \mathrm{Var} \left( W_{\mathbb{T},2}^2 (\mu_N, \lambda) \right) \sim \frac{L^4 c_2(L \alpha)}{\sigma^4} \cdot \frac{\log N}{N^2} \]
with $L = \mathrm{gcd} (\mathrm{supp} (X_1))$, $\sigma^2 = \mathbb{E} (X_1^2)$ and explicit positive constants $c_1(L \alpha)$, $c_2(L \alpha)$. If in addition $\mathbb{E} (|X_1|^p) <\infty$ with some real constant $p>2$, then
\[ \begin{split} \mathbb{E} \left( W_{\mathbb{T},2}^2 (\mu_N, \lambda) \right) &= \frac{L^2 c_1(L \alpha)}{\sigma^2} \cdot \frac{\log N}{N} + O \left( \frac{1}{N} \right) , \\ \mathrm{Var} \left( W_{\mathbb{T},2}^2 (\mu_N, \lambda) \right) &= \frac{L^4 c_2(L \alpha)}{\sigma^4} \cdot \frac{\log N}{N^2} + O \left( \frac{(\log \log N)^2}{N^2} \right) \end{split} \]
with implied constants depending only on $\alpha$ and the distribution of $X_1$.
\end{thm}

In the terminology of probability theory, the sequence $S_n \alpha \pmod{\mathbb{Z}}$ is a Markov chain on $\mathbb{T}$ with stationary measure $\lambda$, started from a nonstationary distribution, indeed from the Dirac measure at $0$. The most interesting feature of this sequence is that it is at the critical threshold of a phase transition between weakly and strongly dependent sequences, which explains why the rates of $\mathbb{E} (W_{\mathbb{T},2}^2 (\mu_N, \lambda))$ and $\mathrm{Var} (W_{\mathbb{T},2}^2 (\mu_N, \lambda))$ differ from uniformly distributed i.i.d.\ samples by a factor of $\log N$. The phase transition is discussed in more detail in Section \ref{latticesection}.

The connection to number theory comes from the values of the constants $c_1$ and $c_2$ in Theorem \ref{maintheorem}. For instance, their values at the golden ratio are
\[ c_1 \left( \frac{1+\sqrt{5}}{2} \right) = \frac{\sqrt{5}}{150 \log \frac{1+\sqrt{5}}{2}} \quad \textrm{and} \quad c_2 \left( \frac{1+\sqrt{5}}{2} \right) = \frac{\sqrt{5}}{6750 \log \frac{1+\sqrt{5}}{2}} . \]
Several more examples are listed in Table \ref{c1c2table}. Finding the explicit values of $c_1 (\alpha)$ and $c_2(\alpha)$ is far from easy, and relies on the theory of binary quadratic forms due to Gauss, as well as arithmetic properties of rings of algebraic integers of real quadratic fields.

A real number $\alpha$ is called a quadratic irrational if it is the root of an irreducible polynomial of degree 2 with integer coefficients. There is a unique minimal polynomial $ax^2+bx+c$ with $a,b,c \in \mathbb{Z}$, $\mathrm{gcd}(a,b,c)=1$ and $a>0$. Then the discriminant $D=b^2-4ac>0$ is not a perfect square, and $\alpha$ is an element of the real quadratic field $\mathbb{Q}(\sqrt{D})$. Rather surprisingly, it turns out that the constants $c_1(\alpha)$ and $c_2(\alpha)$ can be expressed in terms of special values of zeta functions $\zeta(A,s)$ of $\mathbb{Z}$-modules $A$ in $\mathbb{Q}(\sqrt{D})$. We refer to Section \ref{quadraticsection} for definitions and further context.
\begin{thm}\label{c1c2theorem} Let $\alpha$ be a quadratic irrational, and let $D$ be the discriminant of its minimal polynomial. Let $A=\mathbb{Z} + \mathbb{Z}\alpha$, and let $\varepsilon$ be the smallest totally positive unit in the order $\{ \xi \in \mathbb{Q}(\sqrt{D}) : \xi A \subseteq A \}$ such that $\varepsilon>1$. Then
\[ c_1(\alpha) = \frac{2\zeta(A,-1)}{D^{1/2} \log \varepsilon} \quad \textrm{and} \quad c_2(\alpha) = \frac{4 \zeta(A,-3)}{9D^{3/2} \log \varepsilon} . \]
\end{thm}

\begin{table}[t]
\centering
\begin{tabular}{|Sc||Sc|Sc|Sc|Sc|Sc|}
\hline
$\alpha$ & $\sqrt{2}$ & $\sqrt{3}$ & $\sqrt{5}$ & $\sqrt{6}$ & $\sqrt{7}$ \\
\hline
$c_1(\alpha)$ & $\frac{\sqrt{2}}{48 \log (1+\sqrt{2})}$ & $\frac{\sqrt{3}}{36 \log (2+\sqrt{3})}$ & $\frac{\sqrt{5}}{100 \log \frac{1+\sqrt{5}}{2}}$ & $\frac{\sqrt{6}}{24 \log (5+2\sqrt{6})}$ & $\frac{\sqrt{7}}{21 \log(8+3\sqrt{7})}$ \\
\hline
$c_2(\alpha)$ & $\frac{11\sqrt{2}}{17280 \log (1+\sqrt{2})}$ & $\frac{23\sqrt{3}}{19440 \log(2+\sqrt{3})}$ & $\frac{43\sqrt{5}}{54000 \log \frac{1+\sqrt{5}}{2}}$ & $\frac{29\sqrt{6}}{8640 \log (5+2\sqrt{6})}$ & $\frac{113\sqrt{7}}{26460 \log (8+3\sqrt{7})}$ \\
\hline
\end{tabular}

\vspace{5mm}

\begin{tabular}{|Sc||Sc|Sc|Sc|Sc|}
\hline
$\alpha$ & $\sqrt{10}$ & $\sqrt{11}$ & $\sqrt{13}$ & $\sqrt{14}$ \\
\hline
$c_1(\alpha)$ & $\frac{47\sqrt{10}}{1200 \log (3+\sqrt{10})}$ & $\frac{7\sqrt{11}}{132 \log (10+3\sqrt{11})}$ & $\frac{\sqrt{13}}{52 \log \frac{3+\sqrt{13}}{2}}$ &$\frac{5\sqrt{14}}{84 \log (15+4\sqrt{14})}$ \\
\hline
$c_2(\alpha)$ & $\frac{2897\sqrt{10}}{432000 \log (3+\sqrt{10})}$ & $\frac{2153\sqrt{11}}{261360 \log (10+3\sqrt{11})}$ & $\frac{1247\sqrt{13}}{365040 \log \frac{3+\sqrt{13}}{2}}$ & $\frac{2503\sqrt{14}}{211680 \log (15+4\sqrt{14})}$ \\
\hline
\end{tabular}
\caption{The values of $c_1(\alpha)$ and $c_2(\alpha)$ for some quadratic irrational $\alpha$}
\label{c1c2table}
\end{table}

Theorem \ref{c1c2theorem} provides a way to explicitly compute the constants $c_1(\alpha)$ and $c_2(\alpha)$ for any quadratic irrational $\alpha$. Indeed, $\varepsilon$ can be expressed in terms of the smallest positive solution of Pell's equation $x^2-Dy^2=4$, whereas the special values of $\zeta(A,s)$ at negative integers $s$ are rational, and were first found by Shintani \cite{SH}, and later by Zagier \cite{ZA2} using a backward continued fraction algorithm. As an immediate corollary, we see that if $\alpha \in \mathbb{Q}(\sqrt{d})$ with some squarefree integer $d>1$ and $\eta>1$ is the fundamental unit in the ring of algebraic integers, then $c_1(\alpha), c_2(\alpha) \in (\sqrt{d}/ \log \eta) \mathbb{Q}$. We will also show that $c_1$ and $c_2$, as functions defined on the set of quadratic irrationals, satisfy the functional equations
\begin{equation}\label{c1c2functionalequation}
c_j (\alpha) = c_j (\alpha +1) = c_j (-\alpha) = c_j (1/\alpha), \qquad j=1,2.
\end{equation}
In particular, $c_1(\alpha)$ and $c_2(\alpha)$ are invariant under shifts of both the regular and the backward continued fraction expansion of $\alpha$.

Quadratic irrationals play a special role in the ergodic theory of circle rotations as well, as shown by Beck \cite{BE} and Avila, Duryev, Dolgopyat and Sarig \cite{ADDS}. In particular, the same constant $c_1(\alpha)$ appears in probabilistic limit theorems for certain Birkhoff sums known as temporal limit theorems, in lattice point counting in polygons \cite{BE}, and in the $L^2$ discrepancy of 2-dimensional lattices \cite{BO}. A constant very similar to $c_1(\alpha)$ appears in \cite{ADDS}, where a completely different method is given for finding its explicit value using dynamical systems and Markov chains.

The main goal of the present paper is to extend the framework of \cite{ADDS,BE} from ergodic theory to random walks. In particular, Theorems \ref{maintheorem} and \ref{c1c2theorem} are the first results in the literature connecting random walks and the arithmetic of real quadratic fields, and also the first distributional results where special zeta values at $s=-3$ play a role. The choice of the Wasserstein metric in our results is for convenience, we expect other statistics of the empirical measure $\mu_N$ to be related to the arithmetic of $\mathbb{Q}(\alpha)$ as well.

The rest of the paper is organized as follows. Section \ref{contextsection} is an overview of the Wasserstein metric and random walks on $\mathbb{T}$ for further context. Section \ref{badlyapproxsection} deals with the probability theory part of the proof, using exponential sums and Diophantine approximation. We will show that the estimates
\[ \mathbb{E} \left( W_{\mathbb{T},2}^2 (\mu_N, \lambda) \right) \asymp \frac{\log N}{N} \quad \textrm{and} \quad \mathrm{Var} \left( W_{\mathbb{T},2}^2 (\mu_N, \lambda) \right) \asymp \frac{\log N}{N^2} \]
actually hold for any badly approximable $\alpha$, but in general they cannot be refined to precise asymptotic relations. In particular, in Theorem \ref{maintheorem} the condition that $\alpha$ is quadratic irrational cannot be relaxed to $\alpha$ being badly approximable.

In Section \ref{quadraticsection}, we finish the proof of Theorem \ref{maintheorem} and prove Theorem \ref{c1c2theorem} using the theory of binary quadratic forms and real quadratic fields. We discuss the algorithm of Zagier for finding the special values of $\zeta(A,s)$ at negative integers $s$, leading to the explicit values of $c_1(\alpha)$ and $c_2(\alpha)$. As the algorithm of Zagier is somewhat cumbersome, in some special cases we also show how to avoid working with $\mathbb{Z}$-modules, and how to express $c_1(\alpha)$ and $c_2(\alpha)$ in terms of more easily computable special values of Dedekind zeta functions $\zeta_{\mathbb{Q}(\sqrt{d})}(s)$.

\section{Context}\label{contextsection}

\subsection{Notation}

The support of the distribution of an integer-valued random variable $X$ is denoted by $\mathrm{supp} (X)=\{ n \in \mathbb{Z} \, : \, \Pr (X=n)>0 \}$. We call $X$ nondegenerate if there does not exist an integer $n$ such that $X=n$ a.s., that is, if $\mathrm{supp}(X)$ has cardinality at least 2. Let $\mathrm{gcd}(A)$ be the greatest common divisor of a finite or infinite set $A \subseteq \mathbb{Z}$. In particular, $\mathrm{gcd}(\mathrm{supp} (X))$ is the greatest positive integer that divides $X$ a.s. The expected value, variance and covariance are denoted by $\mathbb{E}$, $\mathrm{Var}$ and $\mathrm{Cov}$, respectively. The distance to the nearest integer function is denoted by $\| x \| = \min_{n \in \mathbb{Z}} |x-n|$, $x \in \mathbb{R}$.

We write $a_N \ll b_N$ or $a_N=O(b_N)$ if there exists a constant $C>0$ such that $|a_N| \le C b_N$ for all $N \ge 3$. We write $a_N \asymp b_N$ if $b_N \ll a_N \ll b_N$. Finally, we write $a_N \sim b_N$ if $\lim_{N \to \infty} a_N/b_N=1$. Similar notations are used for functions.

\subsection{Diophantine approximation and continued fractions}

We refer to Khinchin \cite{KH} for a general introduction to Diophantine approximation and regular continued fractions. A real number $\alpha$ is called badly approximable if there exists a constant $C>0$ such that $\| q \alpha \| \ge C/q$ for all $q \in \mathbb{N}$. Every nonzero integer multiple of a badly approximable number is also badly approximable. The set of all badly approximable numbers has Lebesgue measure 0, and Hausdorff dimension 1.

Every irrational number $\alpha$ has a unique regular continued fraction expansion
\[ \alpha=[a_0;a_1,a_2,\ldots] = a_0+ \cfrac{1}{a_1+\cfrac{1}{a_2+\cdots}} \]
with an integer $a_0$ and integers $a_k \ge 1$, $k \ge 1$. An irrational number $\alpha$ is badly approximable if and only if the sequence $a_k$ is bounded. A theorem of Euler and Lagrange states that an irrational number $\alpha$ is a quadratic irrational if and only if the sequence $a_k$ is eventually periodic. In particular, every quadratic irrational is badly approximable.

Every real number $\alpha$ also has a unique backward continued fraction expansion
\[ \alpha = \llbracket b_0, b_1, b_2, \ldots \rrbracket = b_0 - \cfrac{1}{b_1-\cfrac{1}{b_2-\cdots}} \]
with an integer $b_0$ and integers $b_k \ge 2$, $k \ge 1$. We have $b_k=2$ for all $k \ge k_0$ with some $k_0$ if and only if $\alpha$ is rational. The sequence $b_k$ is eventually periodic if and only if $\alpha$ is rational or quadratic irrational. We refer to Hirzebruch \cite{HI} and Zagier \cite{ZA1,ZA2,ZA3} for basic properties and applications of backward continued fractions.

\subsection{Probability metrics}

Let $(M,\rho)$ be a compact metric space, and let $\mathcal{P}(M)$ denote the set of Borel probability measures on $M$. The Wasserstein metric $W_{M,p}$ of order $1 \le p \le \infty$ is a metric on $\mathcal{P}(M)$ defined for finite $p$ as
\[ W_{M,p} (\mu, \nu) = \inf_{\vartheta \in \mathrm{Coupling}(\mu, \nu)} \left( \int_{M \times M} \rho (x,y)^p \, \mathrm{d}\vartheta (x,y) \right)^{1/p}, \quad \mu, \nu \in \mathcal{P}(M) , \]
and for $p=\infty$ (with the $L^p(\vartheta)$-norm replaced by the $L^{\infty}(\vartheta)$-norm) as
\[ W_{M,\infty} (\mu, \nu) = \inf_{\vartheta \in \mathrm{Coupling}(\mu, \nu)} \inf \{ K \ge 0 \, : \, \rho (x,y) \le K \,\, \vartheta-\mathrm{a.e.} \} , \quad \mu, \nu \in \mathcal{P}(M), \]
where $\mathrm{Coupling}(\mu, \nu)$ is the set of $\vartheta \in \mathcal{P}(M \times M)$ whose marginals are $\vartheta (A \times M) = \mu(A)$ and $\vartheta(M \times A) = \nu (A)$, $A \subseteq M$ Borel. The Wasserstein metric originates in the theory of optimal transport, and plays an important role in a large number of fields including PDEs, fluid dynamics and computer science. We refer to Villani \cite{VI} for a comprehensive survey of the theory and applications of optimal transport and the Wasserstein metric.

We only work with two simple compact metric spaces: the unit interval $I=[0,1]$ endowed with the Euclidean metric $\rho(x,y)=|x-y|$, and the circle group $\mathbb{T}=\mathbb{R}/\mathbb{Z}$ endowed with the circular metric $\rho (x,y) = \| x-y \|$. We write $W_{I,p}$ resp.\ $W_{\mathbb{T},p}$ for the Wasserstein metric on $I$ resp.\ $\mathbb{T}$. The space $\mathbb{T}$ can be identified with the interval $[0,1)$ via the fractional part function $x \mapsto \{ x \}$, as well as with the unit circle in $\mathbb{C}$ via $x \mapsto e^{2 \pi i x}$. In particular, this gives a natural identification between $\{ \mu \in \mathcal{P}(I) \, : \, \mu (\{ 1 \})=0 \}$ and $\mathcal{P}(\mathbb{T})$, under which $W_{\mathbb{T},p}(\mu, \nu) \le W_{I,p}(\mu, \nu)$. It will cause no confusion to denote both the Lebesgue measure on $I$ and the normalized Haar measure on $\mathbb{T}$ by $\lambda$.

The optimal transport problem is particularly simple on the 1-dimensional spaces $I$ and $\mathbb{T}$. The main consequence is that the Wasserstein metrics $W_{I,p}$ and $W_{\mathbb{T},p}$ can be expressed in terms of the cumulative distribution functions (cdf's). We define the cdf\footnote{The half-open interval $[0,x)$ in the definition leads to a left-continuous $F_{\mu}$ with $F_{\mu}(0)=0$. The equally suitable choice $[0,x]$ would lead to a right-continuous $F_{\mu}$ with $F_{\mu}(1)=1$.} of $\mu \in \mathcal{P}(I)$ as $F_{\mu}(x) = \mu([0,x))$, $x \in [0,1]$. Graham \cite{GR} showed that for any $\mu \in \mathcal{P}(I)$,
\begin{equation}\label{WIpexplicit}
W_{I,p} (\mu, \lambda) = \| F_{\mu} - F_{\lambda} \|_{L^p ([0,1])}
\end{equation}
is simply the distance of the cdf's in $L^p$-norm. In particular,
\[ W_{I,\infty} (\mu, \lambda) = \sup_{x \in [0,1]} |F_{\mu} (x) - F_{\lambda}(x)| \]
is simply the Kolmogorov metric. These are also known as $L^p$ discrepancy resp.\ $L^{\infty}$ discrepancy in the literature, especially when $\mu$ is the empirical measure of a finite point set. We similarly define the cdf of $\mu \in \mathcal{P}(\mathbb{T})$ as $F_{\mu}(x)=\mu([0,x))$, $x \in [0,1]$, where $[0,x)$ is meant as an interval under the identification of $\mathbb{T}$ with $[0,1)$. In this case \cite{GR}, for any $\mu \in \mathcal{P}(\mathbb{T})$,
\begin{equation}\label{WTpexplicit}
W_{\mathbb{T},p} (\mu, \lambda) = \inf_{c \in \mathbb{R}} \| F_{\mu} - F_{\lambda} -c \|_{L^p ([0,1])} .
\end{equation}
Clearly, $F_{\lambda}(x)=x$ in both $I$ and $\mathbb{T}$.

In the special case $p=2$, we also have explicit formulas for $W_{I,2}(\mu, \lambda)$ and $W_{\mathbb{T},2}(\mu, \lambda)$ in terms of the Fourier coefficients $\hat{\mu}(m) = \int_0^1 e^{2 \pi i m x} \, \mathrm{d}\mu(x)$, $m \in \mathbb{Z}$. Indeed, assuming $\mu(\{ 1 \})=0$, integration by parts shows that the Fourier coefficients of the cdf are $\int_0^1 F_{\mu} (x) e^{2 \pi i m x} \, \mathrm{d}x = (1-\hat{\mu}(m)) / (2 \pi i m)$, $m \neq 0$, hence by the Parseval formula,
\[ W_{I,2}^2(\mu, \lambda) = \sum_{\substack{m \in \mathbb{Z} \\ m \neq 0}} \frac{|\hat{\mu}(m)|^2}{4 \pi^2 m^2} + \left( \int_0^1 (F_{\mu}(x) -x)^2 \, \mathrm{d}x \right)^2 . \]
Noting that the infimum in \eqref{WTpexplicit} is attained at $c=\int_0^1 (F_{\mu}(x) -x) \, \mathrm{d}x$, we also deduce
\begin{equation}\label{WT2explicit}
W_{\mathbb{T},2}^2 (\mu, \lambda) = \sum_{\substack{m \in \mathbb{Z} \\ m \neq 0}} \frac{|\hat{\mu}(m)|^2}{4 \pi^2 m^2} .
\end{equation}
The right-hand side of \eqref{WT2explicit} is also known as diaphony, and further coincides with the so-called periodic $L^2$ discrepancy, defined analogously to the $L^2$ discrepancy, but using all circular arcs on the unit circle instead of arcs anchored at $0$, see \cite{LE}.

We formulate all our results in terms of the Wasserstein metrics $W_{I,p}$ and $W_{\mathbb{T},p}$. As the above discussion shows, these are in fact natural and well known measures of equidistribution defined in terms of cdf's.

\subsection{General random walks on the circle group}

Let us give an overview of random walks on the circle with a general distribution. Let $X_1, X_2, \ldots$ be i.i.d.\ random variables taking values from $\mathbb{T}$, each with distribution $\mu \in \mathcal{P}(\mathbb{T})$. Then $S_n = \sum_{k=1}^n X_k$ is a random walk on $\mathbb{T}$, where addition is meant in the group $\mathbb{T}$, that is, modulo 1. We are interested in the $n$-fold convolution power $\mu^{*n}$ (the distribution of $S_n$) as well as the empirical measure $\mu_N = N^{-1} \sum_{n=1}^N \delta_{S_n}$, where $\delta_x$ denotes the Dirac measure concentrated at the point $x \in \mathbb{T}$. Both converge to the normalized Haar measure $\lambda$ on $\mathbb{T}$ under very general assumptions. L\'evy \cite{LY} showed that $\mu^{*n} \to \lambda$ weakly if and only if $\mathrm{supp}(\mu)$ is not contained in a translate of a finite cyclic subgroup of $\mathbb{T}$. A theorem of Robbins \cite{RO} states that $\Pr (\mu_N \to \lambda \textrm{ weakly})=1$ if and only if $\mathrm{supp}(\mu)$ is not contained in a finite cyclic subgroup of $\mathbb{T}$. The relation $\Pr (\mu_N \to \lambda \textrm{ weakly})=1$ equivalently means that given any continuous function $f: \mathbb{T} \to \mathbb{R}$, the additive functional $\sum_{n=1}^N f(S_n)$ satisfies the strong law of large numbers $N^{-1} \sum_{n=1}^N f(S_n) \to \int_{\mathbb{T}} f \, \mathrm{d} \lambda$ a.s.

The rate of convergence in the relations $\mu^{*n} \to \lambda$ and $\mu_N \to \lambda$ in various probability metrics was first estimated by Schatte \cite{SCH1,SCH2,SCH3,SCH4}, and more recently in \cite{BBR,BB1,BB2,BB4,BR,BW,HS,SU,WE}. See \cite{BB3,BB5,BW,CZ} and references therein for limit theorems for additive functionals $\sum_{n=1}^N f(S_n)$.

We first of all mention that $W_{I,\infty}(\mu^{*n}, \lambda) \to 0$ exponentially fast if and only if $\sup_{m \neq 0} |\hat{\mu}(m)|<1$, which is the analogue of Cram\'er's condition on the circle \cite{SCH1}. In particular, $W_{I,p}(\mu^{*n}, \lambda)$ and $W_{\mathbb{T},p}(\mu^{*n}, \lambda)$ converge to zero exponentially fast for all $1 \le p\le \infty$ whenever $X_1$ has an absolutely continuous component.

Berkes and the author \cite{BB4} proved the following functional central limit theorem and functional law of the iterated logarithm for the cdf $F_{\mu_N}$ of the empirical measure. Let $f_x (t)=\mathds{1}_{[0,x)}(t) - x$, $t \in \mathbb{T}$ denote the centered indicator of the arc $[0,x) \subseteq \mathbb{T}$, let $U$ be a random variable uniformly distributed on $\mathbb{T}$, independent of $X_1, X_2, \ldots$, and let
\[ \Gamma (x,y) = \mathbb{E} (f_x (U) f_y (U)) + \sum_{n=1}^{\infty} \mathbb{E} (f_x (U) f_y (U+S_n)) + \sum_{n=1}^{\infty} \mathbb{E} (f_y (U) f_x (U+S_n)), \quad x,y \in [0,1] . \]
Under the assumption
\begin{equation}\label{psicondition}
W_{I,\infty} (\mu^{*n}, \lambda) = \sup_{x \in [0,1]} |\Pr (S_n \in [0,x)) -x| \ll \frac{1}{n^{1+\delta}} \quad \textrm{with some constant} \quad \delta>0,
\end{equation}
the function $\Gamma$ is continuous, symmetric and positive semidefinite on the unit square $[0,1]^2$, and
\[ \sqrt{N} (F_{\mu_N} - F_{\lambda}) \overset{d}{\to} Y \]
in the Skorokhod space $D[0,1]$, where $Y(x)$, $x \in [0,1]$ is a Gaussian process with mean zero and covariance function $\Gamma$. Further, with probability 1, the sequence
\[ \frac{\sqrt{N}(F_{\mu_N} - F_{\lambda})}{\sqrt{2 \log \log N}}, \quad N \in \mathbb{N} \]
is relatively compact in $D[0,1]$, and its class of limit functions in the Skorokhod topology is the closed unit ball $B(\Gamma)$ of the reproducing kernel Hilbert space defined by the kernel function $\Gamma$. We refer to Billingsley \cite{BI} for a general introduction to functional limit theorems and Skorokhod spaces, and to Aronszajn \cite{AR} for reproducing kernel Hilbert spaces.

These functional limit theorems immediately yield limit theorems for the distance from $\mu_N$ to $\lambda$ in various probability metrics. In particular, formulas \eqref{WIpexplicit} and \eqref{WTpexplicit}, together with the fact that the maps $F \mapsto \| F \|_{L^p ([0,1])}$ and $F \mapsto \inf_{c \in \mathbb{R}} \| F-c \|_{L^p ([0,1])}$ are continuous in the Skorokhod space $D[0,1]$, show the following.
\begin{cor}\label{limitlawscorollary} Under assumption \eqref{psicondition}, for any $1 \le p\le \infty$, we have the limit laws
\[ \sqrt{N} W_{I,p}(\mu_N,\lambda) \overset{d}{\to} \| Y \|_{L^p ([0,1])} \quad \textrm{and} \quad \sqrt{N} W_{\mathbb{T},p}(\mu_N,\lambda) \overset{d}{\to} \inf_{c \in \mathbb{R}} \| Y -c \|_{L^p ([0,1])}, \]
and the laws of the iterated logarithm
\[ \begin{split} \limsup_{N \to \infty} \frac{\sqrt{N}W_{I,p} (\mu_N, \lambda)}{\sqrt{2 \log \log N}} &= \sup_{f \in B(\Gamma)} \| f \|_{L^p ([0,1])} \quad \textrm{a.s.}, \\ \limsup_{N \to \infty} \frac{\sqrt{N}W_{\mathbb{T},p} (\mu_N, \lambda)}{\sqrt{2 \log \log N}} &= \sup_{f \in B(\Gamma)} \inf_{c \in \mathbb{R}} \| f -c \|_{L^p ([0,1])} \quad \textrm{a.s.} \end{split} \]
\end{cor}

If $X_1$ is uniformly distributed on $\mathbb{T}$, then one readily checks that the sequence $S_n$ is also i.i.d.\ uniformly distributed on $\mathbb{T}$. In this case $\Gamma (x,y) = \mathbb{E} (f_x(U) f_y(U)) = \min \{ x,y \} -xy$, hence $Y$ is the Brownian bridge, and
\[ B(\Gamma) = \left\{ f: [0,1] \to \mathbb{R} \, : \, f \textrm{ is absolutely continuous, } f(0)=f(1)=0, \,\, \int_0^1 (f'(x))^2 \, \mathrm{d}x \le 1 \right\} . \]
Corollary \ref{limitlawscorollary} thus contains the empirical measure of an i.i.d.\ uniformly distributed sample as a special case. We refer to Bobkov and Ledoux \cite{BL} for a comprehensive survey on the empirical measure of i.i.d.\ samples with more general distributions on the real line.

\subsection{Lattice random walks on the circle group}\label{latticesection}

The situation is more delicate for random walks with a lattice distribution. In this setting, let $X_1, X_2, \ldots$ be i.i.d.\ integer-valued random variables and $\alpha$ an irrational number, set $S_n = \sum_{k=1}^n X_k$, and consider the lattice random walk $S_n \alpha \pmod{\mathbb{Z}}$ on $\mathbb{T}$. The rate of convergence of both the convolution power $\mu^{*n}$ (the distribution of $S_n \alpha \pmod{\mathbb{Z}}$) and the empirical measure $\mu_N = N^{-1} \sum_{n=1}^N \delta_{S_n \alpha \pmod{\mathbb{Z}}}$ is sensitive to the Diophantine approximation properties of $\alpha$, and the growth rate of the random integer sequence $S_n$. The rate of convergence of $\mu^{*n}$ was basically settled in \cite{BB1}.

Let $\varphi(x)=\mathbb{E} (e^{ixX_1})$ be the characteristic function of $X_1$. Wu and Zhu \cite{WZ} proved that under the assumptions
\begin{equation}\label{varphicondition}
|1-\varphi(x)| \asymp |x|^{\beta} \quad \textrm{for all } x \textrm{ in an open neighborhood of } 0
\end{equation}
and
\begin{equation}\label{diophantinecondition}
0< \liminf_{q \to \infty} q^{\gamma} \| q \alpha \| < \infty
\end{equation}
with some constants $0< \beta \le 2$ and $\gamma \ge 1$, we have for all $1 \le p < \infty$,
\[ \begin{array}{cc} \displaystyle{\mathbb{E} \left( W_{\mathbb{T},p} (\mu_N, \lambda) \right) \asymp \frac{1}{\sqrt{N}}} & \textrm{if } \beta \gamma <2, \\ \displaystyle{\frac{1}{\sqrt{N}} \ll \mathbb{E} \left( W_{\mathbb{T},p} (\mu_N, \lambda) \right) \ll \frac{(\log N)^{\max \{ 1/2, 1-1/p \}}}{\sqrt{N}}} & \textrm{if } \beta \gamma =2, \\ \displaystyle{0<\limsup_{N \to \infty} N^{1/(\beta \gamma)} \mathbb{E} \left( W_{\mathbb{T},p} (\mu_N, \lambda) \right) <\infty} & \textrm{if } \beta \gamma >2. \end{array} \]
We thus observe a phase transition as the parameter $\beta \gamma$ changes, with the case $\beta \gamma <2$ corresponding to weak dependence (matching the typical rate $N^{-1/2}$ for i.i.d.\ samples), and the case $\beta \gamma >2$ corresponding to strong dependence of the sequence $S_n \alpha \pmod{\mathbb{Z}}$. A similar phase transition was observed in \cite{BB2} in the almost sure order of $W_{I,\infty}(\mu_N, \lambda)$, which was also shown to be $N^{-1/\max \{ 2,\beta \gamma \}}$ up to logarithmic factors.

Assuming \eqref{varphicondition}, \eqref{diophantinecondition} and some additional technical assumptions, we have $W_{I,\infty}(\mu^{*n}, \lambda) \ll n^{-1/(\beta \gamma)}$, and this is sharp \cite{BB1}. In particular, in the case $\beta \gamma <1$, condition \eqref{psicondition} is satisfied, and Corollary \ref{limitlawscorollary} yields precise limit theorems for $W_{I,p}(\mu_N,\lambda)$ and $W_{\mathbb{T},p}(\mu_N, \lambda)$ for all $1 \le p\le \infty$. It remains open whether these limit theorems can be extended to $\beta \gamma <2$.

The case when $X_1$ has mean zero and finite, positive variance, and $\alpha$ is badly approximable is particularly interesting: then \eqref{varphicondition} and \eqref{diophantinecondition} are satisfied with $\beta=2$ and $\gamma =1$, thus we are at the critical threshold $\beta \gamma =2$ of the phase transition. In this special case, Wu and Zhu \cite{WZ} were able to remove the logarithmic gap from their estimates, and showed that for $1 \le p \le 2$, we have $\mathbb{E} (W_{\mathbb{T},p} (\mu_N, \lambda)) \asymp \sqrt{(\log N)/N}$. In this paper, we improve their result by finding the precise asymptotics of the expected value and the variance of $W_{\mathbb{T},2}^2(\mu_N, \lambda)$ for quadratic irrational $\alpha$.

We refer to \cite{AN,BOB1,BOB2,DF} for further connections between Diophantine approximation and random walks on the real line, in particular Edgeworth expansions.

\section{Random walks with badly approximable numbers}\label{badlyapproxsection}

For the rest of the paper, let $X_1, X_2, \ldots$ be i.i.d.\ integer-valued random variables with characteristic function $\varphi (x) = \mathbb{E} (e^{ixX_1})$. Fix a real number $\alpha$, let $S_n=\sum_{k=1}^n X_k$, and let $\mu_N = N^{-1} \sum_{n=1}^N \delta_{S_n \alpha \pmod{\mathbb{Z}}}$ denote the empirical measure.

The main result of this section concerns a nondegenerate $X_1$ and a badly approximable $\alpha$.
\begin{thm}\label{badlyapproxtheorem} Assume $X_1$ is nondegenerate, and $\alpha$ is badly approximable. Then
\[ \begin{split} \mathbb{E} \left( W_{\mathbb{T},2}^2 (\mu_N, \lambda) \right) &= \frac{1}{N} \sum_{1 \le m \le \sqrt{N}} \frac{1}{2 \pi^2 m^2} \cdot \frac{1-|\varphi (2 \pi m \alpha)|^2}{|1-\varphi (2 \pi m \alpha)|^2} + O \left( \frac{1}{N} \right) , \\ \mathrm{Var} \left( W_{\mathbb{T},2}^2 (\mu_N, \lambda) \right) &= \frac{1}{N^2} \sum_{1 \le m \le \sqrt{N}} \frac{1}{4 \pi^4 m^4} \left( \frac{1-|\varphi (2 \pi m \alpha)|^2}{|1-\varphi (2 \pi m \alpha)|^2} \right)^2 + O \left( \frac{(\log \log N)^2}{N^2} \right)  \end{split} \]
with implied constants depending only on $\alpha$ and the distribution of $X_1$.
\end{thm}

Theorem \ref{badlyapproxtheorem} will immediately imply the following.
\begin{cor}\label{badlyapproxcorollary} Assume $X_1$ is nondegenerate, and $\alpha$ is badly approximable. Then
\[ \mathbb{E} \left( W_{\mathbb{T},2}^2 (\mu_N, \lambda) \right) \ll \frac{\log N}{N} \quad \textrm{and} \quad \mathrm{Var} \left( W_{\mathbb{T},2}^2 (\mu_N, \lambda) \right) \ll \frac{\log N}{N^2} . \]
If in addition $\mathbb{E} (X_1) =0$ and $0<\mathbb{E} (X_1^2) < \infty$, then
\[ \mathbb{E} \left( W_{\mathbb{T},2}^2 (\mu_N, \lambda) \right) \asymp \frac{\log N}{N} \quad \textrm{and} \quad \mathrm{Var} \left( W_{\mathbb{T},2}^2 (\mu_N, \lambda) \right) \asymp \frac{\log N}{N^2} . \]
All implied constants depend only on $\alpha$ and the distribution of $X_1$.
\end{cor}

The estimates for the expected value in Corollary \ref{badlyapproxcorollary} were proved by Wu and Zhu \cite{WZ}, whereas the estimates for the variance are new. We also show that Corollary \ref{badlyapproxcorollary} cannot be improved to precise asymptotic relations for a general badly approximable $\alpha$.
\begin{cor}\label{noasymptoticcorollary} Assume $\mathbb{E} (X_1) =0$ and $0< \mathbb{E} (X_1^2) < \infty$. Then there exists a badly approximable $\alpha$ for which the sequences
\[ \frac{N}{\log N} \mathbb{E} \left( W_{\mathbb{T},2}^2 (\mu_N, \lambda) \right) \quad \textrm{and} \quad \frac{N^2}{\log N} \mathrm{Var} \left( W_{\mathbb{T},2}^2 (\mu_N, \lambda) \right) \]
do not converge.
\end{cor}

The rest of Section \ref{badlyapproxsection} is dedicated to the proof of Theorem \ref{badlyapproxtheorem} and Corollaries \ref{badlyapproxcorollary} and \ref{noasymptoticcorollary}.

\subsection{Diophantine sum estimates}

We start with some preliminary estimates of sums involving the distance to the nearest integer function.
\begin{lem}\label{diophantinelemma1} Let $\alpha$ be a badly approximable number, and let $1<\theta< \theta'$ be constants. For any $0< \delta \le 1/2$ and any integer $M \ge 1$,
\begin{align}
\sum_{\substack{m=1 \\ \| m \alpha \|<\delta}}^{\infty} \frac{1}{m^{\theta}} &\ll \delta^{\theta}, \label{diophantine1} \\
\sum_{\substack{m=1 \\ \| m \alpha \| \ge \delta}}^{\infty} \frac{1}{m^{\theta} \| m \alpha \|^{\theta'}} &\ll \frac{1}{\delta^{\theta'-\theta}}, \label{diophantine2} \\
\sum_{\substack{m=M \\ \| m \alpha \| \ge \delta}}^{\infty} \frac{1}{m^{\theta} \| m \alpha \|^{\theta}} &\ll \left\{ \begin{array}{cc} 1+\log \frac{1}{M \delta} & \textrm{if } M \delta \le 1, \\ \frac{1}{(M\delta)^{\theta -1}} & \textrm{if } M \delta >1, \end{array} \right. \label{diophantine3} \\
\sum_{m=1}^M \frac{1}{m^{\theta} \| m \alpha \|^{\theta}} &\asymp \log M \label{diophantine4}
\end{align}
with implied constants depending only on $\alpha$, $\theta$ and $\theta'$.
\end{lem}

\begin{proof} Let $C>0$ be a constant such that $\| q \alpha \| \ge C/q$ for all $q \in \mathbb{N}$.

We first prove \eqref{diophantine1}. Given any integer $m \ge 1$, for any two distinct integers $j_1, j_2 \in [1,m]$, we have $\| j_1 \alpha - j_2 \alpha \| \ge C/|j_1-j_2| \ge C/m$. An application of the pigeonhole principle shows that
\[ s_m =  \sum_{j=1}^m \mathds{1}_{\{ \| j \alpha \| < \delta \}} \le 1+\frac{2\delta m}{C} . \]
Observe also that $s_m=0$ for all $1 \le m \le C/\delta$. Summation by parts thus leads to the desired estimate
\[ \sum_{m=1}^{\infty} \frac{1}{m^{\theta}} \mathds{1}_{\{ \| m \alpha \| < \delta \}} = \sum_{m=1}^{\infty} \left( \frac{1}{m^{\theta}} - \frac{1}{(m+1)^{\theta}} \right) s_m \ll \sum_{m>C/\delta} \frac{1}{m^{\theta+1}}(1+\delta m) \ll \delta^{\theta} . \]

We now prove \eqref{diophantine2}. Given an integer $k \ge 0$, consider the points $m \alpha \pmod{\mathbb{Z}}$, $2^k \le m < 2^{k+1}$ on the circle $\mathbb{T}$, which we now identify with the interval $(-1/2,1/2]$. The interval $(-C/2^{k+1}, C/2^{k+1})$ does not contain any of these points, whereas each interval $[Cj/2^{k+1}, C(j+1)/2^{k+1})$, $j \ge 1$ and $(-C(j+1)/2^{k+1}, -Cj/2^{k+1}]$, $j \ge 1$ contains at most one such point. Therefore,
\[ \begin{split} \sum_{2^k \le m < 2^{k+1}} \frac{1}{m^{\theta} \| m \alpha \|^{\theta'}} \mathds{1}_{\{ \| m \alpha \| \ge \delta \}} &\le 2 \sum_{j=1}^{\infty} \frac{1}{2^{\theta k} (Cj/2^{k+1})^{\theta'}} \mathds{1}_{\{ C(j+1)/2^{k+1} \ge \delta \}} \ll 2^{(\theta' - \theta)k} \sum_{j \gg 2^k \delta} \frac{1}{j^{\theta'}} \\ &\ll \left\{ \begin{array}{cc} 2^{(\theta'-\theta)k} & \textrm{if } 2^k \delta \le 1, \\ 2^{(1-\theta)k} \delta^{1-\theta'} & \textrm{if } 2^k \delta >1 , \end{array} \right. \end{split} \]
and
\[ \sum_{m=1}^{\infty} \frac{1}{m^{\theta} \| m \alpha \|^{\theta'}} \mathds{1}_{\{ \| m \alpha \| \ge \delta \}} \ll \sum_{2^k \le 1/\delta} 2^{(\theta' - \theta)k} + \sum_{2^k>1/\delta} 2^{(1-\theta)k} \delta^{1-\theta'} \ll \frac{1}{\delta^{\theta'-\theta}}, \]
as claimed.

To see \eqref{diophantine3}, we similarly deduce
\[ \begin{split} \sum_{2^k \le m < 2^{k+1}} \frac{1}{m^{\theta} \| m \alpha \|^{\theta}} \mathds{1}_{\{ \| m \alpha \| \ge \delta \}} &\le 2 \sum_{j=1}^{\infty} \frac{1}{2^{\theta k} (Cj/2^{k+1})^{\theta}} \mathds{1}_{\{ C(j+1)/2^{k+1} \ge \delta \}} \ll \sum_{j \gg 2^k \delta} \frac{1}{j^{\theta}} \\ &\ll \left\{ \begin{array}{cc} 1 & \textrm{if } 2^k \delta \le 1, \\ 2^{(1-\theta)k} \delta^{1-\theta} & \textrm{if } 2^k \delta>1, \end{array} \right. \end{split} \]
and then sum over all $k \ge 0$ such that $2^{k+1}>M$.

Finally, we prove \eqref{diophantine4}. We similarly deduce
\begin{equation}\label{mthetamalphatheta}
\sum_{2^k \le m < 2^{k+1}} \frac{1}{m^{\theta} \| m \alpha \|^{\theta}} \le 2 \sum_{j=1}^{\infty} \frac{1}{2^{\theta k} (Cj/2^{k+1})^{\theta}} \ll \sum_{j=1}^{\infty} \frac{1}{j^{\theta}} \ll 1.
\end{equation}
Summing over all $0 \le k \le (\log M)/\log 2$ proves the upper bound in \eqref{diophantine4}. To see the lower bound, recall that Dirichlet's approximation theorem \cite[p.\ 41]{KH} states that for any real number $Q \ge 1$ there exists an integer $1 \le q \le Q$ such that $\| q \alpha \|<1/Q$. Since $\alpha$ is badly approximable, this integer also satisfies $q>CQ$. In particular, for any integer $k \ge 0$, there exists an integer $q \in (1/C^k, 1/C^{k+1}]$ such that $\| q \alpha \| < C^{k+1}$, and consequently $q \| q \alpha \| <1$. There are $\gg \log M$ pairwise disjoint intervals $(1/C^k, 1/C^{k+1}]$ contained in the interval $[1,M]$, hence there are $\gg \log M$ terms $m \in [1,M]$ for which $m^{-\theta} \| m \alpha \|^{-\theta} > 1$. This proves the lower bound in \eqref{diophantine4}.
\end{proof}

\begin{lem}\label{diophantinelemma2} Let $\alpha$ be a badly approximable number. For any integer $N \ge 3$,
\[ \sum_{\substack{m_1, m_2 =1 \\ m_1 \neq m_2}}^{\infty} \frac{1}{m_1^2 m_2^2} \min \left\{ \frac{1}{\| m_1 \alpha \|^2}, N \right\} \min \left\{ \frac{1}{\| m_2 \alpha \|^2}, N \right\} \min \left\{ \frac{1}{\| (m_1+m_2) \alpha \|^2}, N \right\} \ll N (\log \log N)^2 \]
with an implied constant depending only on $\alpha$. The same holds with $m_1+m_2$ replaced by $m_1-m_2$.
\end{lem}

\begin{proof} Let
\[ A^{\pm}(m_1, m_2) = \frac{1}{m_1^2 m_2^2} \min \left\{ \frac{1}{\| m_1 \alpha \|^2}, N \right\} \min \left\{ \frac{1}{\| m_2 \alpha \|^2}, N \right\} \min \left\{ \frac{1}{\| (m_1 \pm m_2) \alpha \|^2}, N \right\} . \]
Note that $\| m_1 \alpha \| \ge 2 \| m_2 \alpha \|$ implies $\| (m_1 \pm m_2) \alpha \| \ge \| m_1 \alpha \| - \| m_2 \alpha \| \ge \| m_1 \alpha \|/2$. Similarly, $\| m_2 \alpha \| \ge 2 \| m_1 \alpha \|$ implies $\| (m_1 \pm m_2) \alpha \| \ge \| m_2 \alpha \| - \| m_1 \alpha \| \ge \| m_2 \alpha \|/2$.

Estimates \eqref{diophantine1} and \eqref{diophantine2} in Lemma \ref{diophantinelemma1} with $\delta \asymp 1/\sqrt{N}$, $\theta=2$ and $\theta'=4$ yield
\[ \begin{split} \sum_{\substack{m_1, m_2=1 \\ m_1 \neq m_2 \\ \| m_1 \alpha \| < 1/\sqrt{N}, \,\, \| m_2 \alpha \|< 2/\sqrt{N}}}^{\infty} A^{\pm}(m_1, m_2) &\le \sum_{\substack{m_1 =1 \\ \| m_1 \alpha \| < 1/\sqrt{N}}}^{\infty} \frac{1}{m_1^2} \sum_{\substack{m_2 =1 \\ \| m_2 \alpha \| < 2/\sqrt{N}}}^{\infty} \frac{N^3}{m_2^2} \ll N, \\ \sum_{\substack{m_1, m_2=1 \\ \| m_1 \alpha \| < 1/\sqrt{N}, \,\, \| m_2 \alpha \| \ge 2/\sqrt{N}}}^{\infty} A^{\pm}(m_1, m_2) &\le \sum_{\substack{m_1=1 \\ \| m_1 \alpha \| < 1/\sqrt{N}}}^{\infty} \frac{1}{m_1^2} \sum_{\substack{m_2 =1 \\ \| m_2 \alpha \| \ge 2/\sqrt{N}}}^{\infty} \frac{4N}{m_2^2 \| m_2 \alpha \|^4} \ll N. \end{split} \]
The contribution of all terms with $\| m_1 \alpha \|<1/\sqrt{N}$ is thus $\ll N$. By symmetry, the contribution of all terms with $\| m_2 \alpha \|<1/\sqrt{N}$ is also $\ll N$. Two applications of \eqref{diophantine2} in Lemma \ref{diophantinelemma1} with $\theta=2$ and $\theta'=4$, first with $\delta = 2\| m_1 \alpha \|$ and then with $\delta = 1/\sqrt{N}$ show that
\[ \begin{split} \sum_{\substack{m_1, m_2=1 \\ \| m_1 \alpha \| \ge 1/\sqrt{N}, \,\, \| m_2 \alpha \| \ge 1/\sqrt{N} \\ \| m_2 \alpha \| \ge 2 \| m_1 \alpha \|}}^{\infty} A^{\pm} (m_1, m_2) &\le \sum_{\substack{m_1=1 \\ \| m_1 \alpha \| \ge 1/\sqrt{N}}}^{\infty} \frac{1}{m_1^2 \| m_1 \alpha \|^2} \sum_{\substack{m_2=1 \\ \| m_2 \alpha \| \ge 2 \| m_1 \alpha \|}}^{\infty} \frac{4}{m_2^2 \| m_2 \alpha \|^4} \\ &\ll \sum_{\substack{m_1=1 \\ \| m_1 \alpha \| \ge 1/\sqrt{N}}}^{\infty} \frac{1}{m_1^2 \| m_1 \alpha \|^4} \ll N . \end{split} \]
By symmetry, the contribution of all terms with $\| m_1 \alpha \| \ge 1/\sqrt{N}$, $\| m_2 \alpha \| \ge 1/\sqrt{N}$ and $\| m_1 \alpha \| \ge 2 \| m_2 \alpha \|$ is also $\ll N$.

It remains to estimate the sum over the set
\[ H = \left\{ (m_1, m_2) \in \mathbb{N}^2 : m_1 \neq m_2, \,\, \| m_1 \alpha \| \ge 1/\sqrt{N}, \,\, \| m_2 \alpha \| \ge 1/\sqrt{N}, \,\, \frac{1}{2} < \frac{\| m_1 \alpha \|}{\| m_2 \alpha \|} < 2 \right\} . \]
If $m_1, m_2 \le \sqrt{N}/\log N$, then $|m_1 \pm m_2| \le 2\sqrt{N}/\log N$, consequently $\| (m_1 \pm m_2) \alpha \|\gg (\log N)/\sqrt{N}$. An application of \eqref{diophantine4} in Lemma \ref{diophantinelemma1} with $\theta=2$ thus gives
\[ \sum_{\substack{(m_1, m_2) \in H \\ m_1, m_2 \le \sqrt{N}/\log N}} A^{\pm} (m_1, m_2) \ll \sum_{m_1 \le \sqrt{N}/\log N} \frac{1}{m_1^2 \| m_1 \alpha \|^2} \sum_{m_2 \le \sqrt{N}/\log N} \frac{1}{m_2^2 \| m_2 \alpha \|^2} \cdot \frac{N}{(\log N)^2} \ll N . \]
To estimate the terms with $m_1 \le \sqrt{N}/\log N$ and $\sqrt{N}/\log N < m_2 \le \sqrt{N}\log N$, we use the fact that $\| m_1 \alpha \| \asymp \| m_2 \alpha \|$ for $(m_1, m_2) \in H$. Applying first \eqref{diophantine1} in Lemma \ref{diophantinelemma1} with $\delta = 2 \| m_2 \alpha \|$, and then \eqref{mthetamalphatheta} (together with the fact that the interval $[\sqrt{N}/\log N, \sqrt{N}\log N]$ can be covered by $\ll \log \log N$ intervals of the form $[2^k, 2^{k+1}]$) yields
\[ \begin{split} \sum_{\substack{(m_1, m_2) \in H \\ m_1 \le \sqrt{N}/\log N \\ \sqrt{N}/\log N < m_2 \le \sqrt{N} \log N}} A^{\pm} (m_1, m_2) &\le \sum_{\sqrt{N} / \log N < m_2 \le \sqrt{N} \log N} \frac{4}{m_2^2 \| m_2 \alpha \|^4} \sum_{\substack{m_1=1 \\ \| m_1 \alpha \| < 2 \| m_2 \alpha \|}}^{\infty} \frac{N}{m_1^2} \\ &\ll N \sum_{\sqrt{N} / \log N < m_2 \le \sqrt{N} \log N} \frac{1}{m_2^2 \| m_2 \alpha \|^2} \ll N \log \log N . \end{split} \]
Estimate \eqref{diophantine3} in Lemma \ref{diophantinelemma1} with $\theta=2$, $\delta = 1/\sqrt{N}$ and $M= \lceil \sqrt{N} \log N \rceil$, and estimate \eqref{diophantine4} with $\theta=2$ lead to
\[ \sum_{\substack{(m_1, m_2) \in H \\ m_1 \le \sqrt{N}/\log N \\ m_2 > \sqrt{N} \log N}} A^{\pm} (m_1, m_2) \le \sum_{m_1 \le \sqrt{N}/\log N} \frac{1}{m_1^2 \| m_1 \alpha \|^2} \sum_{\substack{m_2 > \sqrt{N} \log N \\ \| m_2 \alpha \| \ge 1/\sqrt{N}}} \frac{N}{m_2^2 \| m_2 \alpha \|^2} \ll N . \]
Adding the previous three formulas shows that the sum of all terms $(m_1, m_2) \in H$ with $m_1 \le \sqrt{N}/\log N$ is $\ll N \log \log N$. By symmetry, the sum of all terms $(m_1, m_2) \in H$ with $m_2 \le \sqrt{N}/\log N$ is also $\ll N \log \log N$. Finally, \eqref{diophantine3} in Lemma \ref{diophantinelemma1} with $\theta=2$, $\delta=1/\sqrt{N}$ and $M = \lceil \sqrt{N}/\log N \rceil$ gives
\[ \sum_{\substack{(m_1, m_2) \in H \\ m_1, m_2 > \sqrt{N}/\log N}} A^{\pm} (m_1, m_2) \le \sum_{\substack{m_1 > \sqrt{N}/\log N \\ \| m_1 \alpha \| \ge 1/\sqrt{N}}} \frac{1}{m_1^2 \| m_1 \alpha \|^2} \sum_{\substack{m_2 > \sqrt{N}/\log N \\ \| m_2 \alpha \| \ge 1/\sqrt{N}}} \frac{N}{m_2^2 \| m_2 \alpha \|^2} \ll N (\log \log N)^2 . \]
Therefore, $\sum_{(m_1, m_2) \in H} A^{\pm}(m_1, m_2) \ll N (\log \log N)^2$, which concludes the proof.
\end{proof} 

\subsection{The moments of an exponential sum}

The following lemma improves \cite[Proposition 2.2]{BB3}. In fact, we will only need it for $p=2$.
\begin{lem}\label{momentlemma} Assume $X_1$ is nondegenerate. There exists a positive integer $A$ such that for any $x \in \mathbb{R}$ with $Ax \not\in \mathbb{Z}$ and any positive integer $p$,
\[ \left| \mathbb{E} \left( \left| \sum_{n=1}^N e^{2 \pi i S_n x} \right|^{2p} \right) - p! N^p \left( \frac{1-|\varphi (2 \pi x)|^2}{|1-\varphi (2 \pi x)|^2} \right)^p \right| \ll \frac{N}{\| Ax \|^{4p-2}} + \frac{N^{p-1}}{\| Ax \|^{2p+2}} \]
with an implied constant depending only on $p$ and the distribution of $X_1$.
\end{lem}

\begin{proof} Since $X_1$ is nondegenerate, there exist distinct integers $n_1, n_2 \in \mathrm{supp} (X_1)$. We will show that $A=4|n_1-n_2|$ satisfies the claim. Let
\[ R=\frac{16}{(\mathbb{E} \left( \| 4 x (X_1-X_2) \|)^2 \right)} . \]
Since $\Pr (X_1-X_2 = |n_1-n_2|)>0$, we have $\mathbb{E} (\| 4x(X_1-X_2) \|) \gg \| Ax \|$, consequently $R \ll 1/\| Ax \|^2$.

In \cite[Proposition 2.2]{BB3} it was proved that
\[ \left| \mathbb{E} \left( \left| \sum_{n=1}^N e^{2 \pi i S_n x} \right|^{2p} \right) - (p!)^2 \binom{N}{p} \left( \frac{1-|\varphi (2 \pi x)|^2}{|1-\varphi (2 \pi x)|^2} \right)^p \right| \le (2pR)^{2p} \max_{0<q<p} \frac{q^{2p-q}N^q}{q! R^{q-1}} + (pR)^{p+1} N^{p-1} . \]
We point out the necessary modifications to improve it to
\begin{equation}\label{modifiedprop2.2}
\left| \mathbb{E} \left( \left| \sum_{n=1}^N e^{2 \pi i S_n x} \right|^{2p} \right) - (p!)^2 \binom{N}{p} \left( \frac{1-|\varphi (2 \pi x)|^2}{|1-\varphi (2 \pi x)|^2} \right)^p \right| \le (2pR)^{2p} \max_{0<q<p} \frac{q^{2p-q-1}N^q}{q! R^q} + (pR)^{p+1} N^{p-1} .
\end{equation}
The proof remains the same up to \cite[Equation (2.14)]{BB3}. With the notation of that paper, at that point in the proof we have given integers $1 \le M <p$ and $1 \le s \le 2p$, and given indices $1=j_1<j_2<\cdots<j_M \le s$. The next step in the proof is to estimate the number of integers $k_1<k_2<\cdots<k_s$ for which $k_{j_1}, \ldots, k_{j_M}$ is fixed, and
\begin{equation}\label{kjequation}
(k_{j_2 -1} - k_{j_1}) + (k_{j_3 -1} - k_{j_2}) + \cdots + (k_{j_M-1} - k_{j_{M-1}}) + (k_s-k_{j_M}) =k.
\end{equation}
In \cite{BB3} this is done by observing that the $s-M$ integers $k_j$, $j \neq j_1, \ldots, j_M$ all lie in the set
\begin{equation}\label{kjintervals}
[k_{j_1}+1, k_{j_1}+k] \cup \cdots \cup [k_{j_M}+1, k_{j_M}+k]
\end{equation}
of cardinality at most $Mk$, hence the number of $k_1<k_2<\cdots<k_s$ satisfying the given constraints is $\le \binom{Mk}{s-M}$. This estimate can in fact be sharpened to $\le \binom{Mk}{s-M-1}$ if $M<s$, and to $\le 1$ if $M=s$ as follows.

The case $M=s$ is trivial, because then all of $k_1, \ldots, k_s$ are fixed. If $M<s$ and $j_M \neq s$, then the $s-M-1$ integers $k_j$, $j \neq j_1, \ldots, j_M,s$ all lie in the set \eqref{kjintervals} of cardinality at most $Mk$, hence the number of possible $k_1<\cdots <k_{s-1}$ is $\le \binom{Mk}{s-M-1}$. For all possible $k_1<\cdots <k_{s-1}$, there is at most one $k_s$ that satisfies \eqref{kjequation}. If $M<s$ and $j_M=s$, then let $j^* \in \{ j_2, \ldots, j_M \}$ be such that $j^*-1 \not\in \{ j_1, \ldots, j_M \}$. Again, the $s-M-1$ integers $k_j$, $j \neq j_1, \ldots, j_M, j^*-1$ all lie in the set \eqref{kjintervals} of cardinality at most $Mk$, hence the number of possible $k_j$, $j \neq j^*-1$ is at most $\binom{Mk}{s-M-1}$. For all possible $k_j$, $j \neq j^*-1$, there is at most one $k_{j^*-1}$ that satisfies \eqref{kjequation}. This proves the estimate $\le \binom{Mk}{s-M-1}$ for the number of $k_1<\cdots<k_s$ satisfying the given constraints.

Most importantly, the exponent of $k$ in $\binom{Mk}{s-M} \ll k^{s-M}$ is improved by 1 to $\binom{Mk}{s-M-1} \ll k^{s-M-1}$. The rest of the proof is identical to the one in \cite{BB3}. Using the improved estimate $\binom{Mk}{s-M-1}$ resp.\ 1 instead of $\binom{Mk}{s-M}$ leads to the improvement \eqref{modifiedprop2.2}.

Up to an implied constant depending on $p$, the maximum over $0<q<p$ in \eqref{modifiedprop2.2} is attained either at $q=1$ or at $q=p-1$. Using also $R \ll 1/\| Ax \|^2$, we thus deduce
\begin{equation}\label{NAxupperbound}
\left| \mathbb{E} \left( \left| \sum_{n=1}^N e^{2 \pi i S_n x} \right|^{2p} \right) - (p!)^2 \binom{N}{p} \left( \frac{1-|\varphi (2 \pi x)|^2}{|1-\varphi (2 \pi x)|^2} \right)^p \right| \ll \frac{N}{\| Ax \|^{4p-2}} + \frac{N^{p-1}}{\| Ax \|^{2p+2}} .
\end{equation}

Let $L=\mathrm{gcd}(\mathrm{supp}(X_1))$. In particular, $X_1$ only attains integer multiples of $L$ with positive probability. Clearly, $\varphi (2 \pi x)$ is a continuous, $1/L$-periodic function, and we have $\varphi (2 \pi x) =1$ if and only if $x$ is an integer multiple of $1/L$. As the characteristic function of a nondegenerate random variable, $\varphi (2 \pi x)$ satisfies the estimate $|1-\varphi (2 \pi x)| \gg x^2$ in an open neighborhood of $x=0$, see \cite[p.\ 14]{PE}. By continuity and periodicity, this immediately implies
\begin{equation}\label{varphiestimate}
|1-\varphi (2 \pi x)| \gg \| Lx \|^2 \qquad \textrm{uniformly in } x \in \mathbb{R} .
\end{equation}
Together with the identity
\begin{equation}\label{varphisumidentity}
\frac{1-|\varphi (2 \pi x)|^2}{|1-\varphi (2 \pi x)|^2} = 1 + \frac{\varphi (2 \pi x)}{1-\varphi (2 \pi x)} + \frac{\overline{\varphi(2 \pi x)}}{1-\overline{\varphi(2 \pi x)}},
\end{equation}
where the overline denotes the complex conjugate, this leads to
\begin{equation}\label{varphifractionestimate}
\frac{1-|\varphi (2 \pi x)|^2}{|1-\varphi (2 \pi x)|^2} \ll \frac{1}{|1-\varphi (2 \pi x)|} \ll \frac{1}{\| Lx \|^2} .
\end{equation}
Clearly, $L \mid A$, hence $\| Ax \| \le (A/L) \| Lx \|$. Since $(p!)^2 \binom{N}{p} = p! N^p + O(N^{p-1})$, the error of replacing $p!^2 \binom{N}{p}$ by $p! N^p$ in \eqref{NAxupperbound} is $\ll N^{p-1}/\| L x \|^{2p} \ll N^{p-1}/\| Ax \|^{2p}$, thus concluding the proof.
\end{proof}

\subsection{Proof of Theorem \ref{badlyapproxtheorem}}

\begin{proof}[Proof of Theorem \ref{badlyapproxtheorem}] Let $L=\mathrm{gcd}(\mathrm{supp}(X_1))$. Replacing $X_1, X_2, \ldots$ by $X_1/L, X_2/L, \ldots$ and $\alpha$ by $L \alpha$ changes neither $S_n \alpha$ nor $\varphi (2 \pi m \alpha)$. We may thus assume that $L=1$. Let $C>0$ be a constant such that $\| q \alpha \| \ge C/q$ for all $q \in \mathbb{N}$.

For the sake of readability, set $W_N = W_{\mathbb{T},2}^2 (\mu_N, \lambda)$. Formula \eqref{WT2explicit} expresses $W_N$ in terms of exponential sums as
\begin{equation}\label{WNexplicit}
W_N = \sum_{m=1}^{\infty} \frac{|\hat{\mu}_N(m)|^2}{2 \pi^2 m^2} = \sum_{m=1}^{\infty} \frac{1}{2 \pi^2 m^2} \left| \frac{1}{N} \sum_{n=1}^N e^{2 \pi i m S_n \alpha} \right|^2 .
\end{equation}

We first estimate the expected value $\mathbb{E} (W_N)$. An application of estimate \eqref{diophantine1} in Lemma \ref{diophantinelemma1} with $\delta=C/\sqrt{N}$ and $\theta=2$ shows that in the explicit formula \eqref{WNexplicit}, the contribution of the terms with $\| m \alpha \| < C/\sqrt{N}$ is negligible:
\[ \sum_{\substack{m=1 \\ \| m \alpha \|<C/\sqrt{N}}}^{\infty} \frac{1}{2 \pi^2 m^2} \left| \frac{1}{N} \sum_{n=1}^N e^{2 \pi i m S_n \alpha} \right|^2 \le \sum_{\substack{m=1 \\ \| m \alpha \|<C/\sqrt{N}}}^{\infty} \frac{1}{m^2} \ll \frac{1}{N} . \]
Therefore,
\begin{equation}\label{EWN1}
\mathbb{E} (W_N) = \sum_{\substack{m=1 \\ \| m \alpha \| \ge C/\sqrt{N}}}^{\infty} \frac{1}{2 \pi^2 m^2} \mathbb{E} \left( \left| \frac{1}{N} \sum_{n=1}^N e^{2 \pi i m S_n \alpha} \right|^2 \right) + O \left( \frac{1}{N} \right) .
\end{equation}
Expanding the square leads to
\begin{equation}\label{Eexponentialsum}
\begin{split} \mathbb{E} \left( \left| \frac{1}{N} \sum_{n=1}^N e^{2 \pi i m S_n \alpha} \right|^2 \right) &= \frac{1}{N^2} \sum_{n_1, n_2=1}^N \mathbb{E} \left( e^{2 \pi i m (S_{n_1} - S_{n_2}) \alpha} \right) \\ &=\frac{1}{N} + \frac{1}{N^2} \sum_{1 \le n_2 < n_1 \le N} \varphi (2 \pi m \alpha)^{n_1-n_2} + \frac{1}{N} \sum_{1 \le n_1<n_2 \le N} \varphi (-2 \pi m \alpha)^{n_2-n_1} \\ &= \frac{1}{N} \cdot \frac{1-|\varphi (2 \pi m \alpha)|^2}{|1-\varphi (2 \pi m \alpha)|^2} - \frac{2}{N^2} \mathrm{Re} \, \frac{\varphi (2 \pi m \alpha) - \varphi (2 \pi m \alpha)^{N+1}}{(1-\varphi (2 \pi m \alpha))^2} . \end{split}
\end{equation}
In the last step we used identity \eqref{varphisumidentity}. An application of estimates \eqref{varphiestimate} and \eqref{diophantine2} in Lemma \ref{diophantinelemma1} with $\delta=C/\sqrt{N}$, $\theta=2$ and $\theta'=4$ shows that
\[ \begin{split} \left| \sum_{\substack{m=1 \\ \| m \alpha \| \ge C/\sqrt{N}}}^{\infty} \frac{1}{2 \pi^2 m^2} \cdot \frac{2}{N^2} \mathrm{Re} \, \frac{\varphi (2 \pi m \alpha) - \varphi (2 \pi m \alpha)^{N+1}}{(1-\varphi (2 \pi m \alpha))^2} \right| &\le \sum_{\substack{m=1 \\ \| m \alpha \| \ge C/\sqrt{N}}}^{\infty} \frac{1}{m^2 N^2|1-\varphi (2 \pi m \alpha)|^2} \\ &\ll \frac{1}{N^2} \sum_{\substack{m=1 \\ \| m \alpha \| \ge C/\sqrt{N}}}^{\infty} \frac{1}{m^2 \| m \alpha \|^4} \ll \frac{1}{N} . \end{split} \]
Formula \eqref{EWN1} thus simplifies to
\[ \mathbb{E} (W_N) = \frac{1}{N} \sum_{\substack{m=1 \\ \| m \alpha \| \ge C/\sqrt{N}}}^{\infty} \frac{1}{2 \pi^2 m^2} \cdot \frac{1-|\varphi (2 \pi m \alpha)|^2}{|1-\varphi (2 \pi m \alpha)|^2} + O \left( \frac{1}{N} \right) . \]
Recalling \eqref{varphisumidentity}, an application of estimates \eqref{varphiestimate} and \eqref{diophantine3} in Lemma \ref{diophantinelemma1} with $\delta = C/\sqrt{N}$, $\theta =2$ and $M=\lfloor \sqrt{N} \rfloor+1$ shows that the contribution of the terms $m>\sqrt{N}$ in the previous formula is negligible:
\[ \frac{1}{N} \sum_{\substack{m>\sqrt{N} \\ \| m \alpha \| \ge C/\sqrt{N}}} \frac{1}{2 \pi^2 m^2} \cdot \frac{1-|\varphi (2 \pi m \alpha)|^2}{|1-\varphi (2 \pi m \alpha)|^2} \ll \frac{1}{N} \sum_{\substack{m>\sqrt{N} \\ \| m \alpha \| \ge C/\sqrt{N}}} \frac{1}{m^2 \| m \alpha \|^2} \ll \frac{1}{N}.  \]
Observing that $\| m \alpha \| \ge C/\sqrt{N}$ holds for all $1 \le m \le \sqrt{N}$, we thus obtain the desired formula
\[ \mathbb{E} (W_N) = \frac{1}{N} \sum_{1 \le m \le \sqrt{N}} \frac{1}{2 \pi^2 m^2} \cdot \frac{1-|\varphi (2 \pi m \alpha)|^2}{|1-\varphi (2 \pi m \alpha)|^2} + O \left( \frac{1}{N} \right) . \]

We now estimate the variance $\mathrm{Var} (W_N)$. Squaring and then taking the expected value in \eqref{WNexplicit} gives
\[ \begin{split} \mathbb{E} (W_N^2) &= \sum_{m=1}^{\infty} \frac{1}{4 \pi^4 m^4} \mathbb{E} \left( \left| \frac{1}{N} \sum_{n=1}^N e^{2 \pi i m S_n \alpha} \right|^4 \right) \\ &\phantom{={}}+ \sum_{\substack{m_1, m_2=1 \\ m_1 \neq m_2}}^{\infty} \frac{1}{4 \pi^4 m_1^2 m_2^2} \cdot \frac{1}{N^4} \sum_{n_1, n_2, n_3, n_4=1}^N \mathbb{E} \left( e^{2 \pi i (m_1 (S_{n_1} - S_{n_2}) + m_2 (S_{n_3} - S_{n_4}))\alpha} \right) . \end{split} \]
On the other hand, taking the expected value and then squaring in \eqref{WNexplicit} leads to
\[ \begin{split} (\mathbb{E} (W_N))^2 &= \sum_{m=1}^{\infty} \frac{1}{4 \pi^4 m^4} \left( \mathbb{E} \left( \left| \frac{1}{N} \sum_{n=1}^N e^{2 \pi i m S_n \alpha} \right|^2 \right) \right)^2 \\ &\phantom{={}}+ \sum_{\substack{m_1, m_2=1 \\ m_1 \neq m_2}}^{\infty} \frac{1}{4 \pi^4 m_1^2 m_2^2} \cdot \frac{1}{N^4} \sum_{n_1, n_2, n_3, n_4=1}^N \mathbb{E} \left( e^{2 \pi i m_1 (S_{n_1} - S_{n_2}) \alpha} \right) \mathbb{E} \left( e^{2 \pi i m_2 (S_{n_3} - S_{n_4})\alpha} \right) . \end{split} \]
Subtracting the previous two formulas yields $\mathrm{Var} (W_N) = U+V$ with
\[ \begin{split} U&= \sum_{m=1}^{\infty} \frac{1}{4 \pi^4 m^4} \mathrm{Var} \left( \left| \frac{1}{N} \sum_{n=1}^N e^{2 \pi i m S_n \alpha} \right|^2 \right) , \\ V&= \sum_{\substack{m_1, m_2=1 \\ m_1 \neq m_2}}^{\infty} \frac{1}{4 \pi^4 m_1^2 m_2^2} \cdot \frac{1}{N^4} \sum_{n_1, n_2, n_3, n_4=1}^N \mathrm{Cov} \left( e^{2 \pi i m_1 (S_{n_1} - S_{n_2}) \alpha}, e^{2 \pi i m_2 (S_{n_3} - S_{n_4}) \alpha} \right) . \end{split} \]

Consider first $U$. Let $A$ be the positive integer from Lemma \ref{momentlemma}, which depends only on the distribution of $X_1$. Note that $A \alpha$ is also badly approximable, and $\| mA \alpha \| \le A \| m \alpha \|$ for all $m \in \mathbb{N}$. An application of \eqref{diophantine1} in Lemma \ref{diophantinelemma1} with $\delta=C/(A\sqrt{N})$ and $\theta=4$ leads to
\[ \sum_{\substack{m=1 \\ \| m A \alpha \| < C/(A\sqrt{N})}}^{\infty} \frac{1}{4 \pi^4 m^4} \mathrm{Var} \left( \left| \frac{1}{N} \sum_{n=1}^N e^{2 \pi i m S_n \alpha} \right|^2 \right) \le \sum_{\substack{m=1 \\ \| m A \alpha \| < C/(A\sqrt{N})}}^{\infty} \frac{1}{m^4} \ll \frac{1}{N^2}, \]
hence
\[  U= \sum_{\substack{m=1 \\ \| m A \alpha \| \ge C/(A\sqrt{N})}}^{\infty} \frac{1}{4 \pi^4 m^4} \mathrm{Var} \left( \left| \frac{1}{N} \sum_{n=1}^N e^{2 \pi i m S_n \alpha} \right|^2 \right) + O \left( \frac{1}{N^2} \right) . \]
Lemma \ref{momentlemma} with $p=2$ and $x=m\alpha$ shows that
\[ \mathbb{E} \left( \left| \frac{1}{N} \sum_{n=1}^N e^{2 \pi i m S_n \alpha} \right|^4 \right) = \frac{2}{N^2} \left( \frac{1-|\varphi (2 \pi m \alpha)|^2}{|1-\varphi (2 \pi m \alpha)|^2} \right)^2 + O \left( \frac{1}{N^3 \| m A \alpha \|^6} \right) . \]
Squaring \eqref{Eexponentialsum} and using estimate \eqref{varphiestimate} gives
\[ \begin{split} \left( \mathbb{E} \left( \left| \frac{1}{N} \sum_{n=1}^N e^{2 \pi i m S_n \alpha} \right|^2 \right) \right)^2 &= \frac{1}{N^2} \left( \frac{1-|\varphi (2 \pi m \alpha)|^2}{|1-\varphi (2 \pi m \alpha)|^2} \right)^2 + O \left( \frac{1}{N^3 \| m \alpha \|^6} + \frac{1}{N^4 \| m \alpha \|^8} \right) . \end{split} \]
Subtracting the previous two formulas leads to
\[ \mathrm{Var} \left( \left| \frac{1}{N} \sum_{n=1}^N e^{2 \pi i m S_n \alpha} \right|^2 \right) = \frac{1}{N^2} \left( \frac{1-|\varphi (2 \pi m \alpha)|^2}{|1-\varphi (2 \pi m \alpha)|^2} \right)^2 + O \left( \frac{1}{N^3 \| m A \alpha \|^6} + \frac{1}{N^4 \| m \alpha \|^8} \right) . \]
An application of \eqref{diophantine2} in Lemma \ref{diophantinelemma1} with $\delta=C/(A\sqrt{N})$, $\theta=4$ and $\theta'=6$ resp.\ $\theta'=8$ shows that
\[ \sum_{\substack{m=1 \\ \| m A \alpha \| \ge C/(A\sqrt{N})}}^{\infty} \frac{1}{m^4} \left( \frac{1}{N^3 \| mA\alpha \|^6} + \frac{1}{N^4 \| m A \alpha \|^8} \right) \ll \frac{1}{N^2}, \]
hence
\[ U= \frac{1}{N^2} \sum_{\substack{m=1 \\ \| m A \alpha \| \ge C/(A\sqrt{N})}}^{\infty} \frac{1}{4 \pi^4 m^4} \left( \frac{1-|\varphi (2 \pi m \alpha)|^2}{|1-\varphi (2 \pi m \alpha)|^2} \right)^2 + O \left( \frac{1}{N^2} \right) . \]
Estimate \eqref{diophantine3} in Lemma \ref{diophantinelemma1} with $\delta=C/(A\sqrt{N})$, $\theta=4$ and $M=\lfloor \sqrt{N} \rfloor+1$ together with \eqref{varphiestimate} shows that the contribution of the terms $m>\sqrt{N}$ is negligible:
\[ \frac{1}{N^2} \sum_{\substack{m>\sqrt{N} \\ \| m A \alpha \| \ge C/(A\sqrt{N})}} \frac{1}{4 \pi^4 m^4} \left( \frac{1-|\varphi (2 \pi m \alpha)|^2}{|1-\varphi (2 \pi m \alpha)|^2} \right)^2 \ll \frac{1}{N^2}  \sum_{\substack{m>\sqrt{N} \\ \| m A \alpha \| \ge C/(A\sqrt{N})}} \frac{1}{m^4 \| m A \alpha \|^4} \ll \frac{1}{N^2} . \]
Observing that $\| m A \alpha \| \ge C/(A\sqrt{N})$ holds for all $1 \le m \le \sqrt{N}$, we thus obtain
\[ U= \frac{1}{N^2} \sum_{1 \le m \le \sqrt{N}} \frac{1}{4 \pi^4 m^4} \left( \frac{1-|\varphi (2 \pi m \alpha)|^2}{|1-\varphi (2 \pi m \alpha)|^2} \right)^2 + O \left( \frac{1}{N^2} \right) . \]

It remains to prove that $V \ll (\log \log N)^2/N^2$, and the claimed estimate for $\mathrm{Var} (W_N)$ will follow. Assume first that $n_1, n_2, n_3, n_4$ are pairwise distinct, and let us write $\{ n_1, n_2, n_3, n_4 \} = \{ k_1, k_2, k_3, k_4 \}$ with $k_1<k_2<k_3<k_4$. There is a unique permutation $\sigma$ of the set $\{ 1,2,3,4 \}$ for which $n_{\sigma(i)}=k_i$, $i=1,2,3,4$. We will estimate
\[ C_{\sigma}(m_1, m_2) = \sum_{\substack{n_1, n_2, n_3, n_4=1 \\ n_{\sigma(1)} < n_{\sigma(2)} < n_{\sigma(3)} < n_{\sigma(4)}}}^N \mathrm{Cov} \left( e^{2 \pi i m_1 (S_{n_1} - S_{n_2}) \alpha}, e^{2 \pi i m_2 (S_{n_3} - S_{n_4}) \alpha} \right) \]
for all $4!=24$ possible permutations $\sigma$. Observe that there are 8 permutations for which the intervals $I=(\min \{ n_1, n_2 \}, \max \{ n_1, n_2 \}]$ and $J=(\min \{ n_3, n_4 \}, \max \{ n_3, n_4 \}]$ are disjoint: these are 1234 and 3412, and the six others obtained by switching $1 \leftrightarrow 2$ or $3 \leftrightarrow 4$. Whenever $I$ and $J$ are disjoint, we have
\begin{equation}\label{covariancezero}
\mathrm{Cov} \left( e^{2 \pi i m_1 (S_{n_1} - S_{n_2}) \alpha}, e^{2 \pi i m_2 (S_{n_3} - S_{n_4}) \alpha} \right) =0
\end{equation}
since the random variables are independent, thus $C_{\sigma}(m_1, m_2)=0$.

For the remaining 16 permutations, we write
\[ \begin{split} m_1 (S_{n_1} - S_{n_2}) + m_2 (S_{n_3} - S_{n_4}) &= \varepsilon_1 S_{k_1} + \varepsilon_2 S_{k_2} + \varepsilon_3 S_{k_3} + \varepsilon_4 S_{k_4}, \\ m_1 (S_{n_1} - S_{n_2}) &= \varepsilon'_1 S_{k_1} + \varepsilon'_2 S_{k_2} + \varepsilon'_3 S_{k_3} + \varepsilon'_4 S_{k_4}, \\ m_2 (S_{n_3} - S_{n_4}) &= \varepsilon''_1 S_{k_1} + \varepsilon''_2 S_{k_2} + \varepsilon''_3 S_{k_3} + \varepsilon''_4 S_{k_4}, \end{split} \]
with coefficients $\varepsilon_j \in \{ \pm m_1, \pm m_2 \}$, $\varepsilon'_j \in \{ \pm m_1,0 \}$, $\varepsilon''_j \in \{ \pm m_2,0 \}$, $j=1,2,3,4$ depending only on $\sigma$. Note that $\sum_{j=1}^4 \varepsilon_j = \sum_{j=1}^4 \varepsilon'_j = \sum_{j=1}^4 \varepsilon''_j=0$. Letting $c_j = \varepsilon_j + \cdots +\varepsilon_4$, as well as $c'_j = \varepsilon'_j + \cdots +\varepsilon'_4$ and $c''_j = \varepsilon''_j + \cdots +\varepsilon''_4$, we deduce a decomposition into a sum of independent random variables:
\[ \begin{split} m_1 (S_{n_1} - S_{n_2}) + m_2 (S_{n_3} - S_{n_4}) &= c_2 \sum_{i=k_1+1}^{k_2} X_i + c_3 \sum_{i=k_2+1}^{k_3} X_i + c_4 \sum_{i=k_3 +1}^{k_4} X_i , \\ m_1 (S_{n_1} - S_{n_2}) &= c'_2 \sum_{i=k_1+1}^{k_2} X_i + c'_3 \sum_{i=k_2+1}^{k_3} X_i + c'_4 \sum_{i=k_3 +1}^{k_4} X_i , \\ m_2 (S_{n_3} - S_{n_4}) &= c''_2 \sum_{i=k_1+1}^{k_2} X_i + c''_3 \sum_{i=k_2+1}^{k_3} X_i + c''_4 \sum_{i=k_3 +1}^{k_4} X_i . \end{split} \]
In particular,
\[ \begin{split} \mathbb{E} \left( e^{2 \pi i (m_1 (S_{n_1} - S_{n_2}) + m_2 (S_{n_3} - S_{n_4})) \alpha} \right) &= \varphi (2 \pi c_2 \alpha)^{k_2-k_1} \varphi (2 \pi c_3 \alpha)^{k_3-k_2} \varphi (2 \pi c_4 \alpha)^{k_4-k_3} , \\ \mathbb{E} \left( e^{2 \pi i m_1 (S_{n_1} - S_{n_2}) \alpha} \right) &= \varphi (2 \pi c'_2 \alpha)^{k_2-k_1} \varphi (2 \pi c'_3 \alpha)^{k_3-k_2} \varphi (2 \pi c'_4 \alpha)^{k_4-k_3} , \\ \mathbb{E} \left( e^{2 \pi i m_2 (S_{n_3} - S_{n_4}) \alpha} \right) &= \varphi (2 \pi c''_2 \alpha)^{k_2-k_1} \varphi (2 \pi c''_3 \alpha)^{k_3-k_2} \varphi (2 \pi c''_4 \alpha)^{k_4-k_3} . \end{split} \]
Therefore,
\[ \begin{split} \sum_{\substack{n_1, n_2, n_3, n_4=1 \\ n_{\sigma(1)} < n_{\sigma(2)} < n_{\sigma(3)} < n_{\sigma(4)}}}^N &\left| \mathbb{E} \left( e^{2 \pi i (m_1 (S_{n_1} - S_{n_2}) + m_2 (S_{n_3} - S_{n_4})) \alpha} \right) \right| \\ &= \sum_{1 \le k_1<k_2<k_3<k_4 \le N} |\varphi (2 \pi c_2 \alpha)|^{k_2-k_1} |\varphi (2 \pi c_3 \alpha)|^{k_3-k_2} |\varphi (2 \pi c_4 \alpha)|^{k_4-k_3} \\ &\le \sum_{1 \le k_1<k_2<k_3 \le N} |\varphi (2 \pi c_2 \alpha)|^{k_2-k_1} |\varphi (2 \pi c_3 \alpha)|^{k_3-k_2} \min \left\{ \frac{2}{1-|\varphi (2 \pi c_4 \alpha)|}, N \right\} \\ &\le N \prod_{j=2}^4 \min \left\{ \frac{2}{1-|\varphi (2 \pi c_j \alpha)|}, N \right\} , \end{split} \]
and similarly,
\[ \begin{split} \sum_{\substack{n_1, n_2, n_3, n_4=1 \\ n_{\sigma(1)} < n_{\sigma(2)} < n_{\sigma(3)} < n_{\sigma(4)}}}^N &\left| \mathbb{E} \left( e^{2 \pi i m_1 (S_{n_1} - S_{n_2}) \alpha} \right) \mathbb{E} \left( e^{2 \pi i m_2 (S_{n_3} - S_{n_4}) \alpha} \right) \right| \\ &\le N \prod_{j=2}^4 \min \left\{ \frac{2}{1-|\varphi (2 \pi c'_j \alpha) \varphi (2 \pi c''_j \alpha)|}, N \right\} . \end{split} \]
Adding the previous two formulas yields
\[ C_{\sigma}(m_1, m_2) \le N \prod_{j=2}^4 \min \left\{ \frac{2}{1-|\varphi (2 \pi c_j \alpha)|}, N \right\} + N \prod_{j=2}^4 \min \left\{ \frac{2}{1-|\varphi (2 \pi c'_j \alpha) \varphi (2 \pi c''_j \alpha)|}, N \right\} . \]
Now let $L'=\mathrm{gcd} (\mathrm{supp} (X_1-X_2))$. Estimate \eqref{varphiestimate} for the nondegenerate integer-valued random variable $X_1-X_2$ with characteristic function $|\varphi|^2$ reads $1-|\varphi (2 \pi x)|^2 \gg \| L'x \|^2$ uniformly in $x \in \mathbb{R}$. Thus
\[ \frac{1}{1-|\varphi (2 \pi c_j \alpha)|} \le \frac{2}{1-|\varphi (2 \pi c_j \alpha)|^2} \ll \frac{1}{\| c_j L' \alpha \|^2}. \]
Using the fact that $c_j=c'_j + c''_j$, we also deduce
\begin{multline*}
\frac{1}{1-|\varphi (2 \pi c'_j \alpha) \varphi (2 \pi c''_j \alpha)|} \\ \le \min \left\{ \frac{1}{1-|\varphi (2 \pi c'_j \alpha)|}, \frac{1}{1-|\varphi (2 \pi c''_j \alpha)|} \right\} \ll \min \left\{ \frac{1}{\| c'_j L' \alpha \|^2}, \frac{1}{\| c''_j L' \alpha \|^2} \right\} \ll \frac{1}{\| c_j L' \alpha \|^2} .
\end{multline*}
Hence
\[ C_{\sigma}(m_1, m_2) \ll N \prod_{j=2}^4 \min \left\{ \frac{1}{\| c_j L' \alpha \|^2}, N \right\} . \]

One readily checks that for the 16 permutations $\sigma$ for which the intervals $I$ and $J$ intersect, the absolute values $|c_2|$, $|c_3|$, $|c_4|$ are either $|m_1|$, $|m_2|$, $|m_1+m_2|$ or $|m_1|$, $|m_2|$, $|m_1-m_2|$, in some order. Indeed,
\begin{align*}
(\varepsilon_1, \varepsilon_2, \varepsilon_3, \varepsilon_4) &= (m_1, m_2, -m_1, -m_2) & &\textrm{and} & (c_2, c_3, c_4) &= (-m_1, -m_1-m_2, -m_2) & &\textrm{if} & \sigma &= 1324, \\  (\varepsilon_1, \varepsilon_2, \varepsilon_3, \varepsilon_4) &=(m_1, m_2, -m_2, -m_1) & &\textrm{and} & (c_2, c_3, c_4) &= (-m_1, -m_1-m_2, -m_1) & &\textrm{if} & \sigma &= 1342.
\end{align*}
The remaining 14 permutations are obtained from $1324$ or $1342$ by switching $1 \leftrightarrow 2$ or $3 \leftrightarrow 4$ or $12 \leftrightarrow 34$, and these simply switch the sign of $m_1$ resp.\ switch the sign of $m_2$ resp.\ switch $m_1 \leftrightarrow m_2$, leading to the given possible values of $|c_j|$. Summing over all 24 permutations $\sigma$ thus leads to
\begin{equation}\label{sumCsigma}
\begin{split} \sum_{\sigma} C_{\sigma} (m_1, m_2) &= \sum_{\substack{n_1, n_2, n_3, n_4=1 \\ \textrm{pairwise distinct}}}^N \mathrm{Cov} \left( e^{2 \pi i m_1 (S_{n_1} - S_{n_2}) \alpha}, e^{2 \pi i m_2 (S_{n_3} - S_{n_4}) \alpha} \right) \\ &\ll N \min \left\{ \frac{1}{\| m_1 L' \alpha \|^2}, N \right\} \min \left\{ \frac{1}{\| m_2 L' \alpha \|^2}, N \right\} \min \left\{ \frac{1}{\| (m_1+m_2) L' \alpha \|^2}, N \right\} \\ &\phantom{\ll{}} + N \min \left\{ \frac{1}{\| m_1 L' \alpha \|^2}, N \right\} \min \left\{ \frac{1}{\| m_2 L' \alpha \|^2}, N \right\} \min \left\{ \frac{1}{\| (m_1-m_2) L' \alpha \|^2}, N \right\} , \end{split}
\end{equation}
and an application of Lemma \ref{diophantinelemma2} gives
\begin{equation}\label{size4}
\sum_{\substack{m_1, m_2=1 \\ m_1 \neq m_2}}^{\infty} \frac{1}{4 \pi^4 m_1^2 m_2^2} \cdot \frac{1}{N^4} \sum_{\substack{n_1, n_2, n_3, n_4=1 \\ \textrm{pairwise distinct}}}^N \mathrm{Cov} \left( e^{2 \pi i m_1 (S_{n_1} - S_{n_2}) \alpha}, e^{2 \pi i m_2 (S_{n_3} - S_{n_4}) \alpha} \right) \ll \frac{(\log \log N)^2}{N^2} .
\end{equation}

Assume next that $n_1, n_2, n_3, n_4$ has 3 distinct values, and let us write $\{ n_1, n_2, n_3, n_4 \} = \{ k_1, k_2, k_3 \}$ with $k_1<k_2<k_3$. There is a unique surjective map $g: \{ 1,2,3,4 \} \to \{ 1,2,3 \}$ for which $n_i=k_{g(i)}$ for all $i=1,2,3,4$. That is, $n_i$ is the $g(i)$th smallest among $n_1, n_2, n_3, n_4$. We will estimate
\[ C_g (m_1, m_2) = \sum_{\substack{n_1, n_2, n_3, n_4=1 \\ n_i \textrm{ is the }g(i)\textrm{th smallest}}}^N \mathrm{Cov} \left( e^{2 \pi i m_1 (S_{n_1} - S_{n_2}) \alpha}, e^{2 \pi i m_2 (S_{n_3} - S_{n_4}) \alpha} \right) \]
for all 36 surjective maps $g$. For instance, $g=g(1)g(2)g(3)g(4)=1213$ corresponds to summing over all $n_1 =n_3<n_2<n_4$. There are 6 surjective maps $g$ for which $g(1)=g(2)$, and another 6 for which $g(3)=g(4)$. In this case, $S_{n_1}-S_{n_2}=0$ resp.\ $S_{n_3}-S_{n_4}=0$, hence the covariance is zero as in \eqref{covariancezero}. There are 8 surjective maps $g$ for which the intervals $I=(\min \{ n_1, n_2 \}, \max \{ n_1, n_2 \}]$ and $J=(\min \{ n_3, n_4 \}, \max \{ n_3, n_4 \}]$ are nonempty and disjoint: this happens when $\sup I = \inf J$ or $\sup J = \inf I$. In these cases the random variables $S_{n_1} - S_{n_2}$ and $S_{n_3} - S_{n_4}$ are independent, hence \eqref{covariancezero} holds again.

For the remaining 16 surjective maps $g$, we similarly write
\[ \begin{split} m_1 (S_{n_1} - S_{n_2}) + m_2 (S_{n_3} - S_{n_4}) &= \varepsilon_1 S_{k_1} + \varepsilon_2 S_{k_2} + \varepsilon_3 S_{k_3}, \\ m_1 (S_{n_1} - S_{n_2}) &= \varepsilon'_1 S_{k_1} + \varepsilon'_2 S_{k_2} + \varepsilon'_3 S_{k_3}, \\ m_2 (S_{n_3} - S_{n_4}) &= \varepsilon''_1 S_{k_1} + \varepsilon''_2 S_{k_2} + \varepsilon''_3 S_{k_3}, \end{split} \]
with coefficients $\varepsilon_j \in \{ \pm m_1, \pm m_2 \}$, $\varepsilon'_j \in \{ \pm m_1,0 \}$, $\varepsilon''_j \in \{ \pm m_2,0 \}$, $j=1,2,3$ depending only on $g$. Let $c_j = \varepsilon_j + \cdots +\varepsilon_3$, as well as $c'_j = \varepsilon'_j + \cdots +\varepsilon'_3$ and $c''_j = \varepsilon''_j + \cdots +\varepsilon''_3$. Following the steps above, we deduce
\[ C_g (m_1, m_2) \ll N \prod_{j=2}^3 \min \left\{ \frac{1}{\| c_j L' \alpha \|^2}, N \right\} . \]

One readily checks that for the 16 surjective maps $g$ for which the intervals $I$ and $J$ are nonempty and intersect, the absolute values $|c_2|$, $|c_3|$ are either $|m_1|$, $|m_1+m_2|$ or $|m_1|$, $|m_1-m_2|$ or $|m_2|$, $|m_1+m_2|$ or $|m_2|$, $|m_1-m_2|$, in some order. Indeed,
\begin{align*}
(\varepsilon_1, \varepsilon_2, \varepsilon_3) &=(m_1+m_2, -m_1, -m_2) & &\textrm{and} & (c_2, c_3) &= (-m_1-m_2, -m_2) & &\textrm{if} & g &=1213, \\  (\varepsilon_1, \varepsilon_2, \varepsilon_3) &= (m_2, m_1, -m_1-m_2) & &\textrm{and} & (c_2, c_3) &=(-m_2, -m_1-m_2) & &\textrm{if} & g &= 2313.
\end{align*}
The remaining 14 surjective maps are obtained from $1213$ or $2313$ by switching $g(1) \leftrightarrow g(2)$ or $g(3) \leftrightarrow g(4)$ or $g(1)g(2) \leftrightarrow g(3)g(4)$, and these simply switch the sign of $m_1$ resp.\ switch the sign of $m_2$ resp.\ switch $m_1 \leftrightarrow m_2$, leading to the given possible values of $|c_j|$. Summing over all 36 surjective maps $g$ thus leads to an upper bound for
\[ \sum_{g} C_g (m_1, m_2) = \sum_{\substack{n_1, n_2, n_3, n_4=1 \\ |\{ n_1, n_2, n_3, n_4 \}|=3}}^N \mathrm{Cov} \left( e^{2 \pi i m_1 (S_{n_1} - S_{n_2}) \alpha}, e^{2 \pi i m_2 (S_{n_3} - S_{n_4}) \alpha} \right) \]
that is even less than the upper bound in \eqref{sumCsigma}, hence
\[ \sum_{\substack{m_1, m_2=1 \\ m_1 \neq m_2}}^{\infty} \frac{1}{4 \pi^4 m_1^2 m_2^2} \cdot \frac{1}{N^4} \sum_{\substack{n_1, n_2, n_3, n_4=1 \\ |\{ n_1, n_2, n_3, n_4 \}|=3}}^N \mathrm{Cov} \left( e^{2 \pi i m_1 (S_{n_1} - S_{n_2}) \alpha}, e^{2 \pi i m_2 (S_{n_3} - S_{n_4}) \alpha} \right) \ll \frac{(\log \log N)^2}{N^2} . \]

As the number of integer vectors $(n_1, n_2, n_3, n_4) \in [1,N]^4$ such that $n_1, n_2, n_3, n_4$ take at most 2 values is $\ll N^2$, we also have
\[ \sum_{\substack{m_1, m_2=1 \\ m_1 \neq m_2}}^{\infty} \frac{1}{4 \pi^4 m_1^2 m_2^2} \cdot \frac{1}{N^4} \sum_{\substack{n_1, n_2, n_3, n_4=1 \\ |\{ n_1, n_2, n_3, n_4 \}|\le 2}}^N \mathrm{Cov} \left( e^{2 \pi i m_1 (S_{n_1} - S_{n_2}) \alpha}, e^{2 \pi i m_2 (S_{n_3} - S_{n_4}) \alpha} \right) \ll \frac{1}{N^2} . \]
The previous two formulas together with \eqref{size4} show that $V \ll (\log \log N)^2/N^2$, thus concluding the proof.
\end{proof}

\subsection{Proof of Corollaries \ref{badlyapproxcorollary} and \ref{noasymptoticcorollary}}

We now deduce Corollaries \ref{badlyapproxcorollary} and \ref{noasymptoticcorollary} from Theorem \ref{badlyapproxtheorem}.

\begin{proof}[Proof of Corollary \ref{badlyapproxcorollary}] Assume $X_1$ is nondegenerate, and $\alpha$ is badly approximable. Let $L=\mathrm{gcd} (\mathrm{supp}(X_1))$. Estimates \eqref{varphifractionestimate} and \eqref{diophantine4} in Lemma \ref{diophantinelemma1} with $\theta=2$ resp.\ $\theta=4$ yield
\[ \begin{split} \sum_{1 \le m \le \sqrt{N}} \frac{1}{2 \pi^2 m^2} \cdot \frac{1-|\varphi (2 \pi m \alpha)|^2}{|1-\varphi (2 \pi m \alpha)|^2} &\ll  \sum_{1 \le m \le \sqrt{N}} \frac{1}{m^2 \| m L \alpha \|^2} \ll \log N, \\ \sum_{1 \le m \le \sqrt{N}} \frac{1}{4 \pi^4 m^4} \left( \frac{1-|\varphi (2 \pi m \alpha)|^2}{|1-\varphi (2 \pi m \alpha)|^2} \right)^2 &\ll \sum_{1 \le m \le \sqrt{N}} \frac{1}{m^4 \| m L \alpha \|^4} \ll \log N . \end{split} \]
The upper bounds $\mathbb{E} (W_{\mathbb{T},2}^2 (\mu_N, \lambda)) \ll (\log N)/N$ and $\mathrm{Var} (W_{\mathbb{T},2}^2 (\mu_N, \lambda)) \ll (\log N)/N^2$ now immediately follow from Theorem \ref{badlyapproxtheorem}.

Assume now in addition that $\mathbb{E} (X_1)=0$ and $0<\mathbb{E} (X_1^2) < \infty$. Then the characteristic function is twice continuously differentiable, and has the asymptotics $\varphi (x) =1-(\sigma^2/2) x^2+o(x^2)$ as $x \to 0$, where $\sigma^2 = \mathbb{E} (X_1^2)$. In particular,
\[ \frac{1-|\varphi (2 \pi x)|^2}{|1-\varphi (2 \pi x)|^2} \sim \frac{1}{\pi^2 \sigma^2 x^2} \quad \textrm{as } x \to 0. \]
The function $(1-|\varphi (2 \pi x)|^2) / |1-\varphi (2 \pi x)|^2$ is $1/L$-periodic and continuous except at integer multiples of $1/L$. Given an arbitrary $\varepsilon>0$, we thus have
\[ \frac{(1-\varepsilon) L^2}{\pi^2 \sigma^2 \| Lx \|^2} - O(1) \le \frac{1-|\varphi (2 \pi x)|^2}{|1-\varphi (2 \pi x)|^2} \le \frac{(1+\varepsilon) L^2}{\pi^2 \sigma^2 \| Lx \|^2} +O(1), \]
and consequently
\[ \begin{split} \frac{(1-\varepsilon)L^2}{2 \pi^4 \sigma^2} \sum_{1 \le m \le \sqrt{N}} \frac{1}{m^2 \| mL\alpha \|^2} -O(1) &\le \sum_{1 \le m \le \sqrt{N}} \frac{1}{2 \pi^2 m^2} \cdot \frac{1-|\varphi (2 \pi m \alpha)|^2}{|1-\varphi (2 \pi m \alpha)|^2} \\ &\le \frac{(1+\varepsilon)L^2}{2 \pi^4 \sigma^2} \sum_{1 \le m \le \sqrt{N}} \frac{1}{m^2 \| mL\alpha \|^2} +O(1) \end{split} \]
with implied constants depending on $\varepsilon$, $\alpha$ and the distribution of $X_1$. As $\varepsilon$ was arbitrary, this implies
\begin{equation}\label{diophantineasymptotics1}
\sum_{1 \le m \le \sqrt{N}} \frac{1}{2 \pi^2 m^2} \cdot \frac{1-|\varphi (2 \pi m \alpha)|^2}{|1-\varphi (2 \pi m \alpha)|^2} \sim \frac{L^2}{2 \pi^4 \sigma^2} \sum_{1 \le m \le \sqrt{N}} \frac{1}{m^2 \| mL\alpha \|^2} \quad \textrm{as } N \to \infty .
\end{equation}
A similar argument leads to
\begin{equation}\label{diophantineasymptotics2}
\sum_{1 \le m \le \sqrt{N}} \frac{1}{4 \pi^4 m^4} \left( \frac{1-|\varphi (2 \pi m \alpha)|^2}{|1-\varphi (2 \pi m \alpha)|^2} \right)^2 \sim \frac{L^4}{4 \pi^8 \sigma^4} \sum_{1 \le m \le \sqrt{N}} \frac{1}{m^4 \| mL\alpha \|^4} \quad \textrm{as } N \to \infty .
\end{equation}
Estimate \eqref{diophantine4} in Lemma \ref{diophantinelemma1} with $\theta=2$ resp.\ $\theta=4$ shows that the sums in \eqref{diophantineasymptotics1} and \eqref{diophantineasymptotics2} are both $\gg \log N$. The lower bounds $\mathbb{E} (W_{\mathbb{T},2}^2 (\mu_N, \lambda))\gg (\log N)/N$ and $\mathrm{Var} (W_{\mathbb{T},2}^2 (\mu_N, \lambda)) \gg (\log N)/N^2$ now immediately follow from Theorem \ref{badlyapproxtheorem}.
\end{proof}

\begin{proof}[Proof of Corollary \ref{noasymptoticcorollary}] Assume $\mathbb{E}(X_1)=0$ and $0<\mathbb{E} (X_1^2) < \infty$, and let $L=\mathrm{gcd} (\mathrm{supp}(X_1))$. We will actually show that there exists a badly approximable $\alpha$ depending only on $L$ for which the sequences $(N/\log N) \mathbb{E} (W_{\mathbb{T},2}^2 (\mu_N, \lambda))$ and $(N^2/\log N) \mathrm{Var} (W_{\mathbb{T},2}^2 (\mu_N, \lambda))$ do not converge. By Theorem \ref{badlyapproxtheorem} and the asymptotic relations \eqref{diophantineasymptotics1} and \eqref{diophantineasymptotics2}, it will be enough to show the existence of a badly approximable $\alpha$ for which the sequences
\begin{equation}\label{divergent}
\frac{1}{\log M} \sum_{m=1}^M \frac{1}{m^2 \| mL \alpha \|^2} \quad \textrm{and} \quad \frac{1}{\log M} \sum_{m=1}^M \frac{1}{m^4 \| mL \alpha \|^4}
\end{equation}
do not converge as $M \to \infty$.

Let $a$ and $\rho$ be large positive integers, and consider the sequence $a_1, a_2, \ldots$ defined as $a_1=1$, and
\[ a_j = \left\{ \begin{array}{ll} 1 & \textrm{if } \rho^{2k} < j \le \rho^{2k+1} \textrm{ with some } k\ge 0, \\ a & \textrm{if } \rho^{2k+1} < j \le \rho^{2k+2} \textrm{ with some } k \ge 0. \end{array} \right. \]
Define $\alpha$ via the regular continued fraction expansion $L \alpha = [0;a_1,a_2,\ldots]$, and let $p_k/q_k = [0;a_1, \ldots, a_k]$ be the convergents. In \cite{BO} we showed that for any $\theta>1$,
\[ \left| \sum_{m=1}^{q_k-1} \frac{1}{m^{\theta} \| m L \alpha \|^{\theta}} - \zeta(2\theta) \sum_{j=1}^k a_j^{\theta} \right| \le 6^{\theta} \frac{4\theta^2}{(\theta-1)^2} \sum_{j=1}^k a_j^{\theta-1} , \]
where $\zeta$ is the Riemann zeta function. With $\theta=2$ resp.\ $\theta=4$ this immediately implies that
\[ \begin{split} \sum_{m=1}^{q_{\rho^{2k+1}}-1} \frac{1}{m^2 \| mL \alpha \|^2} &\ll \sum_{j=1}^{\rho^{2k+1}} a_j^2 \ll \rho^{2k+1}, \\ \sum_{m=1}^{q_{\rho^{2k+1}}-1} \frac{1}{m^4 \| mL \alpha \|^4} &\ll \sum_{j=1}^{\rho^{2k+1}} a_j^4 \ll \rho^{2k+1}, \end{split} \]
and
\[ \begin{split} \sum_{m=1}^{q_{\rho^{2k+2}}-1} \frac{1}{m^2 \| mL \alpha \|^2} &\ge \zeta(2) \sum_{j=1}^{\rho^{2k+2}} a_j^2 - O \bigg( \sum_{j=1}^{\rho^{2k+2}} a_j \bigg) \gg a^2 \rho^{2k+2}, \\ \sum_{m=1}^{q_{\rho^{2k+2}}-1} \frac{1}{m^4 \| mL \alpha \|^4} &\ge \zeta(4) \sum_{j=1}^{\rho^{2k+2}} a_j^4 - O \bigg( \sum_{j=1}^{\rho^{2k+2}} a_j^3 \bigg) \gg a^4 \rho^{2k+2} \end{split} \]
provided that $\rho$ is large enough in terms of $a$. The recursion $q_k=a_k q_{k-1} +q_{k-2}$ shows that $q_k \le (a_k+1)q_{k-1}$. Hence $q_k \le \prod_{j=1}^k (a_j+1)$, and consequently $\log q_k \le \sum_{j=1}^k \log (a_j+1) \ll k \log a$. On the other hand, $q_k$ is at least the $(k+1)$st Fibonacci number, therefore $\log q_k \gg k$. In particular,
\begin{align*}
\frac{1}{\log M} \sum_{m=1}^M \frac{1}{m^2 \| mL \alpha \|^2} &\ll 1 & &\textrm{and} & \frac{1}{\log M} \sum_{m=1}^M \frac{1}{m^4 \| mL \alpha \|^4} &\ll 1 & &\textrm{for } M=q_{\rho^{2k+1}}-1, \\ \frac{1}{\log M} \sum_{m=1}^M \frac{1}{m^2 \| mL \alpha \|^2} &\gg \frac{a^2}{\log a} & &\textrm{and} & \frac{1}{\log M} \sum_{m=1}^M \frac{1}{m^4 \| mL \alpha \|^4} &\gg \frac{a^4}{\log a} & &\textrm{for } M=q_{\rho^{2k+2}}-1 .
\end{align*}
Choosing $a$ large enough, we conclude that the sequences in \eqref{divergent} do not converge as $M \to \infty$.
\end{proof}

\section{Random walks with quadratic irrationals}\label{quadraticsection}

In Lemma \ref{diophantinelemma1}, we showed that for any badly approximable $\alpha$ and any constant $\theta>1$,
\[ \sum_{m=1}^M \frac{1}{m^{\theta} \| m \alpha \|^{\theta}} \asymp \log M . \]
Together with Theorem \ref{badlyapproxtheorem} this led to
\[ \mathbb{E} \left( W_{\mathbb{T},2}^2 (\mu_N, \lambda) \right) \asymp \frac{\log N}{N} \quad \textrm{and} \quad \mathrm{Var} \left( W_{\mathbb{T},2}^2 (\mu_N, \lambda) \right) \asymp \frac{\log N}{N^2} . \]
In Corollary \ref{noasymptoticcorollary}, we also showed that these estimates cannot be improved to precise asymptotic relations for a general badly approximable number.

In the special case of a quadratic irrational $\alpha$, however, it is possible to refine the estimates above to precise asymptotic relations. As the main result of this section, in Theorem \ref{quadraticsumtheorem} below we will show that for any quadratic irrational $\alpha$ and any constant $\theta>1$, we have
\begin{equation}\label{quadraticasymptotics}
\sum_{m=1}^M \frac{1}{m^{\theta} \| m \alpha \|^{\theta}} = c(\alpha, \theta) \log M +O(1) .
\end{equation}
Together with Theorem \ref{badlyapproxtheorem}, the asymptotic relation \eqref{quadraticasymptotics} with $\theta=2$ and $\theta=4$ will lead to the precise asymptotics of $\mathbb{E} (W_{\mathbb{T},2}^2 (\mu_N, \lambda))$ and $\mathrm{Var} (W_{\mathbb{T},2}^2 (\mu_N, \lambda))$ as stated in Theorem \ref{maintheorem}.

Relation \eqref{quadraticasymptotics} with $\theta=2$ was first proved by Beck \cite{BE}, who also computed the explicit value of $c(\alpha,2)$ in the special case when $\mathbb{Q}(\alpha)$ has class number 1. We generalize his result to any $\theta>1$, and show how to compute $c(\alpha,\theta)$ for an arbitrary quadratic irrational $\alpha$.

Sections \ref{quadraticformssection} and \ref{quadraticfieldssection} are quick overviews of binary quadratic forms, real quadratic fields and zeta functions of $\mathbb{Z}$-modules. We show the asymptotic relation \eqref{quadraticasymptotics} in Section \ref{diophantinequadraticsection}. We deduce Theorem \ref{maintheorem} from Theorem \ref{badlyapproxtheorem} and \eqref{quadraticasymptotics} in Section \ref{maintheoremproofsection}. Finally, we show how to compute the constants $c(\alpha, \theta)$, $c_1(\alpha)$ and $c_2(\alpha)$, and prove Theorem \ref{c1c2theorem} in Section \ref{computingsection}.

\subsection{Binary quadratic forms}\label{quadraticformssection}

We refer to Landau \cite{LA} for an introduction to binary quadratic forms and Pell's equation, including the proof of all facts stated in this section.

A binary quadratic form is a polynomial $Q(x,y)=ax^2+bxy+cy^2$ with integer coefficients $a,b,c \in \mathbb{Z}$. The discriminant of $Q$ is defined as $D=b^2-4ac$. We assume throughout that $D \neq 0$. Clearly, $D \equiv 0$ or $1 \pmod{4}$. Let $f$ be the largest positive integer such that $f^2 \mid D$ and $D/f^2 \equiv 0$ or $1 \pmod{4}$. There exists a unique squarefree integer $d$ such that
\begin{equation}\label{discriminantDfd}
D = \left\{ \begin{array}{ll} f^2 d & \textrm{if } d \equiv 1 \pmod{4} , \\ 4 f^2 d & \textrm{if } d \equiv 2 \textrm{ or } 3 \pmod{4} . \end{array} \right.
\end{equation}
The discriminant $D$ is called a fundamental discriminant if $f=1$.

The group of $2 \times 2$ integer matrices with determinant $1$,
\[ \mathrm{SL}(2,\mathbb{Z}) = \left\{ \left( \begin{array}{cc} r & s \\ t & u \end{array} \right) \, : \, r,s,t,u \in \mathbb{Z}, \, ru-st=1 \right\} \]
acts on the set of binary quadratic forms by multiplication of the variables $(x,y)$, thought of as a column vector. That is,
\[ (Q \circ B)(x,y) = Q(rx+sy, tx+uy) \quad \textrm{for} \quad B=\left( \begin{array}{cc} r & s \\ t & u \end{array} \right) \in \mathrm{SL}(2,\mathbb{Z}) . \]
Two binary quadratic forms are called equivalent if they are in the same orbit. One readily checks that $\mathrm{gcd}(a,b,c)$ and the discriminant $D$ are invariant under the group action. The number of orbits (equivalence classes) in the set of all binary quadratic forms with $\mathrm{gcd}(a,b,c)=1$ and discriminant $D$ is called the class number of $D$, and is denoted by $h(D)$.

Fix a form $Q(x,y)=ax^2 + bxy + cy^2$ with $\mathrm{gcd}(a,b,c)=1$ whose discriminant $D>0$ is not a perfect square. The automorphism group $\mathrm{Aut}(Q)=\{ B \in \mathrm{SL}(2,\mathbb{Z}) : Q \circ B = Q \}$ has the explicit description
\[ \mathrm{Aut}(Q) = \left\{ \left( \begin{array}{cc} \frac{t-bu}{2} & -cu \\ au & \frac{t+bu}{2} \end{array} \right) \, : \, t,u \in \mathbb{Z}, \, t^2-Du^2=4 \right\} . \]
If $(t_1,u_1)$ and $(t_2,u_2)$ are integer solutions of Pell's equation $t^2-Du^2=4$, and the integers $t_3, u_3$ are defined via
\[ \frac{t_1+u_1 \sqrt{D}}{2} \cdot \frac{t_2+u_2 \sqrt{D}}{2} = \frac{t_3+u_3\sqrt{D}}{2}, \]
then $(t_3,u_3)$ is another integer solution, and in fact
\[ \left( \begin{array}{cc} \frac{t_1-bu_1}{2} & -cu_1 \\ au_1 & \frac{t_1+bu_1}{2} \end{array} \right) \left( \begin{array}{cc} \frac{t_2-bu_2}{2} & -cu_2 \\ au_2 & \frac{t_2+bu_2}{2} \end{array} \right) = \left( \begin{array}{cc} \frac{t_3-bu_3}{2} & -cu_3 \\ au_3 & \frac{t_3+bu_3}{2} \end{array} \right) . \]
The smallest positive solution of Pell's equation $t^2-Du^2=4$ is defined as the (unique) positive integer solution $(t_0,u_0)$ for which $u_0$ is minimal. Letting $\varepsilon=(t_0+u_0\sqrt{D})/2$, the set of all solutions is given by
\[ \left\{ \frac{t+u\sqrt{D}}{2} \, : \, t,u \in \mathbb{Z}, \, t^2-Du^2=4 \right\} = \{ w \varepsilon^j \, : \, w \in \{1,-1 \}, j \in \mathbb{Z} \} . \]

\begin{table}[t]
\centering
\begin{tabular}{|Sc||Sc|Sc|Sc|Sc|Sc|Sc|Sc|Sc|}
\hline
$D$ & 8 & 12 & 20 & 24 & 28 & 32 & 40 & 44 \\
\hline
$\varepsilon$ & $(1+\sqrt{2})^2$ & $2+\sqrt{3}$ & $\left( \frac{1+\sqrt{5}}{2} \right)^6$ & $5+2\sqrt{6}$ & $8+3\sqrt{7}$ & $(1+\sqrt{2})^2$ & $(3+\sqrt{10})^2$ & $10+3\sqrt{11}$ \\
\hline
\end{tabular}

\vspace{5mm}

\begin{tabular}{|Sc||Sc|Sc|Sc|Sc|Sc|Sc|Sc|}
\hline
$D$ & 5 & 13 & 17 & 21 & 29 & 33 & 37 \\
\hline
$\varepsilon$ & $\left( \frac{1+\sqrt{5}}{2} \right)^2$ & $\left( \frac{3+\sqrt{13}}{2} \right)^2$ & $(4+\sqrt{17})^2$ & $\frac{5+\sqrt{21}}{2}$ & $\left( \frac{5+\sqrt{29}}{2} \right)^2$ & $23+4\sqrt{33}$ & $(6+\sqrt{37})^2$ \\
\hline
\end{tabular}

\caption{The value of $\varepsilon$ for some discriminants, expressed as a power of the fundamental unit}
\label{epsilontable}
\end{table}

The group $\mathrm{Aut}(Q)$ also acts on the integer lattice $\mathbb{Z}^2$ by multiplication from the left. Given a nonzero integer $n$, the preimage $Q^{-1}(n)=\{ (x,y) \in \mathbb{Z}^2 \, : \, Q(x,y)=n \}$ is clearly invariant under this action. Let $R_Q(n)=|Q^{-1}(n) / \mathrm{Aut}(Q)|$ denote the number of orbits in $Q^{-1}(n)$. Clearly, $R_{Q \circ B}(n)=R_Q(n)$ for any $B \in \mathrm{SL}(2,\mathbb{Z})$, so $R_Q(n)$ in fact depends only on the equivalence class of $Q$. Note also $Q^{-1}(-n)=(-Q)^{-1}(n)$ and $\mathrm{Aut}(-Q)=\mathrm{Aut}(Q)$, hence $R_Q(-n)=R_{-Q}(n)$.

There is a canonical way to choose a point from each orbit in $Q^{-1}(n)$ as follows. The set of primary representations is defined as
\[ P_Q = \left\{ (x,y) \in \mathbb{Z}^2 \, : \, 2ax+(b-\sqrt{D})y > 0 \textrm{ and } 1 \le \left| \frac{2ax + (b+\sqrt{D})y}{2ax + (b-\sqrt{D})y} \right| < \varepsilon^2 \right\} . \]
It is not difficult to see that for any $(x,y) \in Q^{-1}(n)$, thought of as a column vector $\binom{x}{y}$, there is a unique $B \in \mathrm{Aut}(Q)$ for which $B \binom{x}{y} \in P_Q$. Indeed, if $B \binom{x}{y}= \binom{x'}{y'}$ where
\[ B = \left( \begin{array}{cc} \frac{t-bu}{2} & -cu \\ au & \frac{t+bu}{2} \end{array} \right) \quad \textrm{with} \quad t^2-Du^2=4, \]
then
\[ \begin{split} 2ax' + (b+\sqrt{D})y' &= \left( 2ax + (b+\sqrt{D})y \right) \frac{t+u\sqrt{D}}{2} , \\ 2ax' + (b-\sqrt{D})y' &= \left( 2ax + (b-\sqrt{D})y \right) \frac{t-u\sqrt{D}}{2} . \end{split} \]
If $(t+u\sqrt{D})/2 = w \varepsilon^j$, then $(t-u\sqrt{D})/2 = w \varepsilon^{-j}$, consequently
\[ (x',y') \in P_Q \, \Longleftrightarrow \, \left( 2ax + (b-\sqrt{D})y \right) w \varepsilon^{-j} >0 \textrm{ and } 1 \le \left|\frac{2ax + (b+\sqrt{D})y}{2ax + (b-\sqrt{D})y} \right| \varepsilon^{2j} < \varepsilon^2 . \]
There is a unique $w \in \{ 1,-1 \}$ that satisfies the first condition, and a unique $j \in \mathbb{Z}$ that satisfies the second one. In particular, each orbit in $Q^{-1}(n)$ intersects $P_Q$ in exactly one point, hence $R_Q(n)=|Q^{-1}(n) \cap P_Q|$ is the number of primary representations of $n$.

It turns out that $R_Q(n)$ is always finite. In fact, it satisfies the following estimate.
\begin{lem}\label{RQnlemma} Let $Q(x,y)=ax^2+bxy+cy^2$, and assume $\mathrm{gcd}(a,b,c)=1$ and $D>0$ is not a perfect square. For any nonzero integer $n$, we have $R_Q(n) \ll |n|^{\delta}$ with an arbitrary constant $\delta>0$, and an implied constant depending only on $D$ and $\delta$.
\end{lem}

\begin{proof} By definition, the number of equivalence classes in the set of all binary quadratic forms with $\mathrm{gcd}(a,b,c)=1$ and discriminant $D$ is $h(D)$. Let $Q_i$, $1 \le i \le h(D)$ be representatives from each equivalence class, and consider $\sum_{i=1}^{h(D)} R_{Q_i}(n)$, that is, the total number of primary representations of $n$ by all binary quadratic forms of discriminant $D$ up to equivalence. We have $R_Q(n)=R_{Q_i}(n)$ for some $1 \le i \le h(D)$, consequently $R_Q(n) \le \sum_{i=1}^{h(D)} R_{Q_i}(n)$. Since $-Q_i$, $1 \le i \le h(D)$ is also a set of representatives from each equivalence class and $R_{Q_i}(-n)=R_{-Q_i}(n)$, the function $\sum_{i=1}^{h(D)}R_{Q_i}(n)$ is even in the variable $n$. Thus it is enough to prove $\sum_{i=1}^{h(D)}R_{Q_i}(n) \ll n^{\delta}$ for positive integers $n$.

The following explicit formula was proved for positive integers $n$ in \cite[Theorem 4.1]{SW}, see also \cite[p.\ 74]{ZA3}. If $\mathrm{gcd}(n,f^2)$ is not a perfect square, then $\sum_{i=1}^{h(D)}R_{Q_i}(n)=0$. If $\mathrm{gcd}(n,f^2)=m^2$ with a positive integer $m$, then
\begin{equation}\label{sumRQiexplicit}
\sum_{i=1}^{h(D)}R_{Q_i}(n) = \left( m \prod_{p \mid m} \left( 1-\frac{1}{p} \left( \frac{D/m^2}{p} \right) \right) \right) \sum_{k \mid \frac{n}{m^2}} \left( \frac{D/f^2}{k} \right),
\end{equation}
where $(\frac{a}{b})$ is the Kronecker symbol. Since $m \mid f$, and the Kronecker symbol only attains the values $-1$, $0$ and $1$, we deduce
\[ \sum_{i=1}^{h(D)}R_{Q_i}(n) \le \left( f \prod_{p \mid f} \left( 1+\frac{1}{p} \right) \right) \tau (n) , \]
where $\tau (n)$ is the number of positive divisors of $n$. Here $f$ depends only on the discriminant, and $\tau(n) \ll n^{\delta}$ with any constant $\delta>0$, which concludes the proof.
\end{proof}

\subsection{Diophantine sums with quadratic irrationals}\label{diophantinequadraticsection}

We now prove \eqref{quadraticasymptotics} with an explicit constant $c(\alpha, \theta)$.
\begin{thm}\label{quadraticsumtheorem} Let $\alpha$ be a quadratic irrational, let $ax^2+bx+c$ be its minimal polynomial with $a,b,c \in \mathbb{Z}$, $\mathrm{gcd}(a,b,c)=1$ and $a>0$, and let $Q(x,y)=ax^2+bxy+cy^2$ be the corresponding binary quadratic form. For any constant $\theta>1$ and any integer $M \ge 1$, we have
\[ \sum_{m=1}^M \frac{1}{m^{\theta} \| m \alpha \|^{\theta}} = c(\alpha, \theta) \log M +O(1) \]
with an implied constant depending only on $\alpha$ and $\theta$. The constant $c(\alpha, \theta)$ is given by
\[ c(\alpha, \theta) = \frac{D^{\theta/2}}{\log \varepsilon} \sum_{\substack{n \in \mathbb{Z} \\ n \neq 0}} \frac{R_Q(n)}{|n|^{\theta}}, \]
where $D=b^2-4ac>0$ is the discriminant, $R_Q(n)$ is the number of primary representations of $n$, and $\varepsilon=(t_0+u_0\sqrt{D})/2$ with the smallest positive solution $(t_0,u_0)$ of Pell's equation $t^2-Du^2=4$.
\end{thm}

\begin{proof} Let $\overline{\alpha}$ be the other root of the minimal polynomial $ax^2+bx+c$. The quadratic form then factors into $Q(x,y)=a(y\alpha -x)(y\overline{\alpha}-x)$. Note that $|\alpha - \overline{\alpha}|=\sqrt{D}/a$.

Fix constants $C_1< \alpha <C_2$ such that $\overline{\alpha} \not\in [C_1,C_2]$. Let $1 \le m \le M$ and $C_1 m \le \ell \le C_2 m$ be integers. We have
\[ \frac{1}{m^{\theta} |m \alpha -\ell|^{\theta}} = \frac{a^{\theta} |m \overline{\alpha} - \ell|^{\theta}}{m^{\theta} a^{\theta}|(m \alpha - \ell)(m \overline{\alpha}-\ell)|^{\theta}} = \frac{a^{\theta} |m \overline{\alpha} -\ell|^{\theta}}{m^{\theta} |Q(\ell,m)|^{\theta}} . \]
Here $||m \overline{\alpha} -\ell| - m|\alpha - \overline{\alpha}|| \le |m \alpha - \ell|$, therefore
\[ |m\overline{\alpha} - \ell|^{\theta} = m^{\theta} |\alpha - \overline{\alpha}|^{\theta} + O(m^{\theta-1} |m \alpha - \ell|) = \frac{m^{\theta} D^{\theta/2}}{a^{\theta}} + O(m^{\theta-1}|m \alpha - \ell|) . \]
The assumption $\overline{\alpha} \not\in [C_1,C_2]$ ensures that $|m \overline{\alpha} - \ell| \gg m$, hence
\[ \frac{1}{m^{\theta} |m \alpha - \ell|^{\theta}} = \frac{D^{\theta/2}}{|Q(\ell, m)|^{\theta}} + O \left( \frac{1}{m^{\theta+1} |m \alpha - \ell|^{\theta-1}} \right) . \]
Note that the integer closest to $m \alpha$ belongs to the interval $[C_1 m , C_2 m]$ for all but $O(1)$ positive integers $m$, thus
\[ \sum_{C_1 m \le \ell \le C_2 m} \frac{1}{m^{\theta} |m \alpha - \ell|^{\theta}} = \frac{1}{m^{\theta} \| m \alpha \|^{\theta}} + O \left( \frac{1}{m^{\theta}} \right) , \]
and
\[ \begin{split} \sum_{C_1 m \le \ell \le C_2 m} \frac{1}{m^{\theta+1} |m \alpha - \ell|^{\theta-1}} &\ll \frac{1}{m^{\theta+1} \| m \alpha \|^{\theta-1}} + \left\{ \begin{array}{ll} 1/m^{2\theta-1} & \textrm{if } \theta <2, \\ (\log m)/m^3 & \textrm{if } \theta=2, \\ 1/m^{\theta+1} & \textrm{if } \theta>2 \end{array} \right. \\ &\ll \frac{1}{m^2} + \frac{1}{m^{2 \theta-1}} . \end{split} \]
In the last step we used $\| m \alpha \| \gg 1/m$. The previous three formulas show that
\[ \frac{1}{m^{\theta} \| m \alpha \|^{\theta}} = \sum_{C_1 m \le \ell \le C_2 m} \frac{D^{\theta/2}}{|Q(\ell,m)|^{\theta}} + O \left( \frac{1}{m^{\theta}} + \frac{1}{m^2} + \frac{1}{m^{2\theta-1}} \right) , \]
and summing over $1 \le m \le M$ leads to
\begin{equation}\label{sumQT}
\sum_{m=1}^M \frac{1}{m^{\theta} \| m \alpha \|^{\theta}} = \sum_{m=1}^M \sum_{C_1 m \le \ell \le C_2 m} \frac{D^{\theta/2}}{|Q(\ell, m)|^{\theta}} + O (1) .
\end{equation}

Now let $n$ be a nonzero integer, and let us estimate the number of solutions to $Q(\ell, m)=n$ with $1 \le m \le M$ and $C_1 m \le \ell \le C_2 m$. Let $\{ (x_k,y_k) : 1 \le k \le R_Q(n) \} = Q^{-1}(n) \cap P_Q$ be the set of primary representations of $n$. By definition, $\xi_k = 2ax_k + (b+\sqrt{D})y_k$ and $\overline{\xi_k} = 2ax_k + (b-\sqrt{D})y_k$ satisfy $\overline{\xi_k}>0$ and $1 \le |\xi_k / \overline{\xi_k}| < \varepsilon^2$. Further, we have
\[ \xi_k \overline{\xi_k} = 4a (ax_k^2 + b x_k y_k + c y_k^2) = 4an, \]
hence $\mathrm{sign}(\xi_k)=\mathrm{sign}(n)$, $|\xi_k| \asymp |n|^{1/2}$ and $\overline{\xi_k} \asymp |n|^{1/2}$. The set of all representations of $n$ is obtained by multiplying the primary representations $(x_k,y_k)$ as column vectors by the elements in $\mathrm{Aut}(Q)$, as explained in Section \ref{quadraticformssection}. We thus have $Q(\ell,m)=n$ if and only if
\[ 2a\ell + (b+\sqrt{D})m = w \varepsilon^j \xi_k \quad \textrm{and} \quad 2a\ell + (b-\sqrt{D})m = w \varepsilon^{-j} \overline{\xi_k} \]
with some $1 \le k \le R_Q(n)$, $w \in \{ 1,-1 \}$ and $j \in \mathbb{Z}$. Fix $k$. Subtracting the equations gives $m = w (\varepsilon^j \xi_k - \varepsilon^{-j} \overline{\xi_k}) / (2 \sqrt{D})$. Up to $O(\log |n|)$ many solutions, this $m$ lies in the interval $[1,M]$ if and only if
\[ w =\mathrm{sign}(n) \textrm{ and } 1 \le j \le \frac{\log M}{\log \varepsilon} ; \qquad \textrm{or} \qquad w =-1 \textrm{ and } -\frac{\log M}{\log \varepsilon} \le j \le -1 . \]
In the first case, $\ell = \frac{-b+\sqrt{D}}{2a} m + O(|n|^{1/2}\varepsilon^{-j})$, whereas in the second case, $\ell =\frac{-b-\sqrt{D}}{2a} m + O(|n|^{1/2}\varepsilon^j)$. Depending on whether $\alpha = \frac{-b+\sqrt{D}}{2a}$ or $\frac{-b-\sqrt{D}}{2a}$, in one of the cases $C_1 m \le \ell \le C_2 m$ holds for all but $O(\log |n|)$ values of $j$, whereas in the other case $C_1 m \le \ell \le C_2 m$ holds only for $O(\log |n|)$ values of $j$. Altogether, we find that the number of integer pairs $(\ell,m)$ such that $1 \le m \le M$, $C_1 m \le \ell \le C_2 m$ and $Q(\ell,m)=n$ is $(R_Q(n) /\log \varepsilon) \log M+O(R_Q(n) \log |n|)$. Therefore,
\[ \sum_{m=1}^M \sum_{C_1 m \le \ell \le C_2 m} \frac{D^{\theta/2}}{|Q(\ell, m)|^{\theta}} = \sum_{\substack{n \in \mathbb{Z} \\ n \neq 0}} \frac{D^{\theta/2}}{|n|^{\theta}} \left( \frac{R_Q(n)}{\log \varepsilon} \log M + O(R_Q(n) \log |n|) \right) . \]
By Lemma \ref{RQnlemma}, $R_Q(n) \ll |n|^{\delta}$ with any $\delta>0$, thus \eqref{sumQT} simplifies to the desired formula
\[ \sum_{m=1}^M \frac{1}{m^{\theta} \| m \alpha \|^{\theta}} = \frac{D^{\theta/2}}{\log \varepsilon} \sum_{\substack{n \in \mathbb{Z} \\ n \neq 0}} \frac{R_Q(n)}{|n|^{\theta}} \log M + O (1). \]
\end{proof}

Note that for any quadratic irrational $\alpha$ and any constant $\theta>1$,
\begin{equation}\label{cfunctionalequation}
c(\alpha, \theta) = c(\alpha+1,\theta) = c(-\alpha, \theta) = c(1/\alpha, \theta) .
\end{equation}
The first two equations are clear from the fact that $\sum_{m=1}^M 1/(m^{\theta} \| m \alpha \|^{\theta})$ is a $1$-periodic and even function of the variable $\alpha$. To see the third equation, note that if $\alpha$ has minimal polynomial $ax^2 + bx+c$ with corresponding binary quadratic form $Q(x,y)=ax^2 + bxy+cy^2$, then $1/\alpha$ has minimal polynomial $cx^2 + bx+a$ with corresponding binary quadratic form $Q'(x,y) = Q(y,x)$. In particular, the discriminant $D$, and consequently $\varepsilon$ are the same for $\alpha$ and $1/\alpha$. Further, for any nonzero integer $n$,
\[ (Q')^{-1} (n) = \left( \begin{array}{cc} 0 & 1 \\ 1 & 0 \end{array} \right) Q^{-1}(n) \quad \textrm{and} \quad \mathrm{Aut}(Q') = \left( \begin{array}{cc} 0 & 1 \\ 1 & 0 \end{array} \right) \mathrm{Aut} (Q) \left( \begin{array}{cc} 0 & 1 \\ 1 & 0 \end{array} \right) . \]
Thus there is a bijection between $(Q')^{-1}(n)/\mathrm{Aut}(Q')$ and $Q^{-1}(n) / \mathrm{Aut} (Q)$. By definition, $R_{Q'}(n)=R_Q(n)$, which finally shows $c(1/\alpha, \theta) = c(\alpha, \theta)$. In particular, $c(\alpha, \theta)$ is invariant under shifts of both the regular and the backward continued fraction expansion of $\alpha$.

\subsection{Proof of Theorem \ref{maintheorem}}\label{maintheoremproofsection}

We now deduce Theorem \ref{maintheorem} from Theorems \ref{badlyapproxtheorem} and \ref{quadraticsumtheorem}.

\begin{proof}[Proof of Theorem \ref{maintheorem}] Assume $\mathbb{E} (X_1) =0$ and $0< \mathbb{E} (X_1^2) < \infty$, and $\alpha$ is a quadratic irrational. Theorem \ref{badlyapproxtheorem} and formulas \eqref{diophantineasymptotics1} and \eqref{diophantineasymptotics2} show that
\[ \mathbb{E} \left( W_{\mathbb{T},2}^2 (\mu_N, \lambda) \right) \sim \frac{L^2}{2 \pi^4 \sigma^2 N} \sum_{1 \le m \le \sqrt{N}} \frac{1}{m^2 \|m L \alpha \|^2}, \quad \mathrm{Var} \left( W_{\mathbb{T},2}^2 (\mu_N, \lambda) \right) \sim \frac{L^4}{4 \pi^8 \sigma^4 N^2} \sum_{1 \le m \le \sqrt{N}} \frac{1}{m^4 \|m L \alpha \|^4} \]
as $N \to \infty$. The claimed relations
\[ \mathbb{E} \left( W_{\mathbb{T},2}^2 (\mu_N, \lambda) \right) \sim \frac{L^2 c_1(L \alpha)}{\sigma^2} \cdot \frac{\log N}{N}, \qquad \mathrm{Var} \left( W_{\mathbb{T},2}^2 (\mu_N, \lambda) \right) \sim \frac{L^4 c_2(L \alpha)}{\sigma^4} \cdot \frac{\log N}{N^2} \]
follow from Theorem \ref{quadraticsumtheorem} with the constants
\[ c_1 (L \alpha) = \frac{c(L \alpha, 2)}{4 \pi^4} \quad \textrm{and} \quad c_2 (L \alpha) = \frac{c(L \alpha, 4)}{8 \pi^8} . \]

Assume now in addition that $\mathbb{E} (|X_1|^p) < \infty$ with some real constant $p>2$. We may assume $p<3$. The characteristic function $\varphi$ of $X_1$ then satisfies $\varphi (x) = 1-(\sigma^2/2)x^2 + O(|x|^p)$ in an open neighborhood of $x=0$. In particular,
\[ \frac{1-|\varphi (2 \pi x)|^2}{|1-\varphi (2 \pi x)|^2} = \frac{1}{\pi^2 \sigma^2 x^2} + O \left( \frac{1}{|x|^{4-p}} \right) \]
in an open neighborhood of $x=0$. By periodicity and continuity,
\[ \frac{1-|\varphi (2 \pi x)|^2}{|1-\varphi (2 \pi x)|^2} = \frac{L^2}{\pi^2 \sigma^2 \| L x \|^2} + O \left( \frac{1}{\| Lx \|^{4-p}} \right) \quad \textrm{uniformly in } x \in \mathbb{R} . \]
Theorem \ref{badlyapproxtheorem} thus gives
\[ \mathbb{E} \left( W_{\mathbb{T},2}^2 (\mu_N, \lambda) \right) = \frac{L^2}{2 \pi^4 \sigma^2 N} \sum_{1 \le m \le \sqrt{N}} \frac{1}{m^2 \| m L \alpha \|^2} + O \left( \frac{1}{N} \sum_{1 \le m \le \sqrt{N}} \frac{1}{m^2 \| m L \alpha \|^{4-p}} \right) . \]
A similar argument shows
\[ \begin{split} \mathrm{Var} \left( W_{\mathbb{T},2}^2 (\mu_N, \lambda) \right) &= \frac{L^4}{4 \pi^8 \sigma^4 N^2} \sum_{1 \le m \le \sqrt{N}} \frac{1}{m^4 \| m L \alpha \|^4} \\ &\phantom{={}}+ O \left( \frac{1}{N^2} \sum_{1 \le m \le \sqrt{N}} \frac{1}{m^4 \| m L \alpha \|^{6-p}} + \frac{(\log \log N)^2}{N^2} \right) . \end{split} \]
Summation by parts and \eqref{diophantine4} in Lemma \ref{diophantinelemma1} with $\theta=4-p$ resp.\ $\theta=6-p$ lead to
\[ \begin{split} \sum_{1 \le m \le \sqrt{N}} \frac{1}{m^2 \| m L \alpha \|^{4-p}} &\ll \sum_{m=1}^{\infty} \left( \frac{1}{m^{p-2}} - \frac{1}{(m+1)^{p-2}} \right) \log m \ll 1, \\ \sum_{1 \le m \le \sqrt{N}} \frac{1}{m^4 \| m L \alpha \|^{6-p}} &\ll \sum_{m=1}^{\infty} \left( \frac{1}{m^{p-2}} - \frac{1}{(m+1)^{p-2}} \right) \log m \ll 1 . \end{split} \]
An application of Theorem \ref{quadraticsumtheorem} together with the previous four formulas finally yield the desired relations
\[ \begin{split} \mathbb{E} \left( W_{\mathbb{T},2}^2 (\mu_N, \lambda) \right) &= \frac{L^2 c_1(L \alpha)}{\sigma^2} \cdot \frac{\log N}{N} + O \left( \frac{1}{N} \right), \\ \mathrm{Var} \left( W_{\mathbb{T},2}^2 (\mu_N, \lambda) \right) &= \frac{L^4 c_2(L \alpha)}{\sigma^4} \cdot \frac{\log N}{N^2} + O \left( \frac{(\log \log N)^2}{N^2} \right) . \end{split} \]
\end{proof}

\subsection{Quadratic fields}\label{quadraticfieldssection}

We refer to Zagier \cite{ZA3} for an introduction to quadratic fields and zeta functions, to Borevich and Shafarevich \cite{BS} for modules and orders in quadratic fields, and to Apostol \cite{AP} for the general theory of Dirichlet characters and $L$-functions.

Let $d>1$ be a squarefree integer, and consider the real quadratic field $\mathbb{Q}(\sqrt{d})$. The conjugate of $\xi=x+y\sqrt{d}$ ($x,y \in \mathbb{Q}$) is denoted by $\overline{\xi}=x-y\sqrt{d}$, and the norm by $N(\xi)=\xi \overline{\xi}$. An element $\xi \in \mathbb{Q}(\sqrt{d})$ is called totally positive if $\xi, \overline{\xi}>0$.

The ring of algebraic integers of $\mathbb{Q}(\sqrt{d})$ is $\mathcal{O}=\mathbb{Z} + \mathbb{Z} \omega_d = \{ x+y\omega_d \, : \, x,y, \in \mathbb{Z} \}$ with
\[ \omega_d = \left\{ \begin{array}{ll} \frac{1+\sqrt{d}}{2} & \textrm{if } d \equiv 1 \pmod{4}, \\ \sqrt{d} & \textrm{if } d \equiv 2 \textrm{ or } 3 \pmod{4} . \end{array} \right. \]
The discriminant of $\mathcal{O}$ is the fundamental discriminant
\[ D_0 = \left\{ \begin{array}{ll} d & \textrm{if } d \equiv 1 \pmod{4}, \\ 4d & \textrm{if } d \equiv 2 \textrm{ or } 3 \pmod{4} . \end{array} \right. \]

A $\mathbb{Z}$-module of rank 2 that is also a ring is called an order. The set of all orders in the field $\mathbb{Q}(\sqrt{d})$ is $\mathcal{O}_f=\mathbb{Z} + \mathbb{Z}f \omega_d$, $f \ge 1$ integer. The discriminant of $\mathcal{O}_f$ is $D=f^2 D_0$. In particular, the maximal order is the ring of algebraic integers $\mathcal{O}=\mathcal{O}_1$.

The set of totally positive units in the ring $\mathcal{O}_f$ is denoted by $U_f^+ = \{ \xi \in \mathcal{O}_f : \xi>0, \, N(\xi)=1 \}$. It is a cyclic group generated by the smallest $\varepsilon \in U_f^+$ such that $\varepsilon>1$. In fact, $\varepsilon=(t_0+u_0\sqrt{D})/2$ with the smallest positive solution $(t_0,u_0)$ of Pell's equation $t^2-Du^2=4$, as before. In particular, $U_f^+=\{ \varepsilon^j : j \in \mathbb{Z} \}$. Note that $\varepsilon$ is a unit also in the ring of algebraic integers $\mathcal{O}$, in particular it is a power of the fundamental unit in $\mathcal{O}$.

Let $A=\mathbb{Z}\xi_1 + \mathbb{Z}\xi_2$ be a $\mathbb{Z}$-module of rank 2 in $\mathbb{Q}(\sqrt{d})$. The discriminant of $A$ is $D(A)=(\xi_1 \overline{\xi_2} - \overline{\xi_1} \xi_2)^2$, and the norm $N(A)$ is defined as the greatest common divisor of $\{ N(\xi) : \xi \in A \}$ (rational, but not necessarily an integer). The set $\{ \xi \in \mathbb{Q}(\sqrt{d}) : \xi A \subseteq A \}$ is an order, hence it is $\mathcal{O}_f$ for some positive interger $f$. We have $D(A)=N(A)^2 D$, where $D=f^2 D_0$ is the discriminant of $\mathcal{O}_f$.

Two $\mathbb{Z}$-modules $A_1$ and $A_2$ are called equivalent if $A_2=\xi A_1$ for some totally positive $\xi \in \mathbb{Q}(\sqrt{d})$. For a fixed nonsquare discriminant $D>0$, the equivalence classes of $\mathbb{Z}$-modules $A$ of rank 2 such that the order $\{ \xi \in \mathbb{Q}(\sqrt{d}) : \xi A \subseteq A \}$ has discriminant $D$, and the equivalence classes of binary quadratic forms $Q(x,y)=ax^2 + bxy+cy^2$ with $\mathrm{gcd}(a,b,c)=1$ and discriminant $D$ are in one-to-one correspondance as follows. The equivalence class of $A=\mathbb{Z}\xi_1 + \mathbb{Z}\xi_2$, with the basis written in an oriented way, that is, $\xi_1 \overline{\xi_2} - \overline{\xi_1} \xi_2>0$, corresponds to the equivalence class of $Q(x,y) = N(x\xi_1 + y \xi_2)/N(A)$. In the reverse direction, the equivalence class of $Q(x,y)=ax^2+bxy+cy^2$ corresponds to the equivalence class of $A=\mathbb{Z}w + \mathbb{Z}w\frac{b-\sqrt{D}}{2a}$, where $w$ is arbitrary with $aN(w)>0$. Note that different choices of $w$ lead to equivalent modules.

\subsubsection{The Dedekind zeta function}

The Dedekind zeta function of $\mathbb{Q}(\sqrt{d})$ is defined as
\[ \zeta_{\mathbb{Q}(\sqrt{d})} (s) = \sum_{0 \neq I \subseteq \mathcal{O}} \frac{1}{N(I)^s}, \qquad \mathrm{Re}\, s >1, \]
where the summation is over all nonzero ideals $I$ in the ring $\mathcal{O}$, and $N(I)=|\mathcal{O} / I|$ is the absolute norm of $I$. Let $(\frac{a}{b})$ denote the Kronecker symbol. The function $\chi_{D_0} (n)=(\frac{D_0}{n})$ is an even, primitive Dirichlet character mod $D_0$. Let $L(s,\chi_{D_0})=\sum_{n=1}^{\infty} \chi_{D_0} (n) n^{-s}$, $\mathrm{Re}\, s>1$ be the corresponding $L$-function. The Dedekind zeta function then factors into $\zeta_{\mathbb{Q}(\sqrt{d})}(s) = \zeta(s) L(s,\chi_{D_0} )$, $\mathrm{Re}\, s>1$, where $\zeta(s)$ is the Riemann zeta function.

All three functions $\zeta(s)$, $\zeta_{\mathbb{Q}(\sqrt{d})}(s)$ and $L(s,\chi_{D_0})$ have meromorphic continuations to $\mathbb{C}$: the former two have a simple pole at $s=1$, whereas $L(s,\chi_{D_0})$ is actually an entire function. All three functions also satisfy functional equations, relating their values at $s$ to their values at $1-s$.

Just like the Riemann zeta function, the Dedekind zeta function also has explicit special values at positive even integers. One possible way of finding them is to first find the special values of $\zeta_{\mathbb{Q}(\sqrt{d})}(s)$ at negative integers \cite[p.\ 51]{ZA3}, and then use the functional equation. For the sake of completeness, we include a proof of the special values at $s=2$ and $s=4$ which avoids the analytic continuation and the functional equation. Table \ref{zetatable} lists the values for some squarefree integers $d>1$.

\begin{table}[t]
\centering
\begin{tabular}{|Sc||Sc|Sc|Sc|Sc|Sc|Sc|Sc|Sc|}
\hline
$d$ & 2 & 3 & 5 & 6 & 7 & 10 & 11 & 13 \\
\hline
$\zeta_{\mathbb{Q}(\sqrt{d})}(2)$ & $\frac{\sqrt{2}\pi^4}{96}$ & $\frac{\sqrt{3}\pi^4}{108}$ & $\frac{2\sqrt{5}\pi^4}{375}$ & $\frac{\sqrt{6}\pi^4}{144}$ & $\frac{\sqrt{7}\pi^4}{147}$ & $\frac{7\sqrt{10} \pi^4}{1200}$ & $\frac{7\sqrt{11}\pi^4}{1452}$ & $\frac{2\sqrt{13}\pi^4}{507}$ \\
\hline
$\zeta_{\mathbb{Q}(\sqrt{d})}(4)$ & $\frac{11 \sqrt{2} \pi^8}{138240}$ & $\frac{23\sqrt{3}\pi^8}{349920}$ & $\frac{4\sqrt{5}\pi^8}{84375}$ & $\frac{29\sqrt{6}\pi^8}{622080}$ & $\frac{113\sqrt{7}\pi^8}{2593080}$ & $\frac{1577\sqrt{10}\pi^8}{43200000}$ & $\frac{2153\sqrt{11}\pi^8}{63249120}$ & $\frac{116\sqrt{13}\pi^8}{3855735}$ \\
\hline
\end{tabular}

\vspace{5mm}

\begin{tabular}{|Sc||Sc|Sc|Sc|Sc|Sc|Sc|Sc|}
\hline
$d$ & 14 & 15 & 17 & 19 & 21 & 22 & 23 \\
\hline
$\zeta_{\mathbb{Q}(\sqrt{d})}(2)$ & $\frac{5\sqrt{14}\pi^4}{1176}$ & $\frac{\sqrt{15}\pi^4}{225}$ & $\frac{4\sqrt{17}\pi^4}{867}$ & $\frac{\sqrt{19}\pi^4}{228}$ & $\frac{4\sqrt{21}\pi^4}{1323}$ & $\frac{23\sqrt{22}\pi^4}{5808}$ & $\frac{5\sqrt{23}\pi^4}{1587}$ \\
\hline
$\zeta_{\mathbb{Q}(\sqrt{d})}(4)$ & $\frac{2503\sqrt{14}\pi^8}{82978560}$ & $\frac{179\sqrt{15}\pi^8}{6075000}$ & $\frac{328\sqrt{17}\pi^8}{11275335}$ & $\frac{14933\sqrt{19}\pi^8}{562986720}$ & $\frac{88\sqrt{21}\pi^8}{3750705}$ & $\frac{24889\sqrt{22}\pi^8}{1011985920}$ & $\frac{7093\sqrt{23}\pi^8}{302228280}$ \\
\hline
\end{tabular}
\caption{The values of $\zeta_{\mathbb{Q}(\sqrt{d})}(2)$ and $\zeta_{\mathbb{Q}(\sqrt{d})}(4)$ for all squarefree integers $2 \le d \le 23$}
\label{zetatable}
\end{table}

\begin{lem}\label{dedekindat24lemma} For any squarefree integer $d>1$,
\[ \begin{split} \zeta_{\mathbb{Q}(\sqrt{d})} (2) &= \frac{\pi^4}{6 D_0^{5/2}} \sum_{k=1}^{D_0-1} \chi_{D_0} (k) k^2 , \\ \zeta_{\mathbb{Q}(\sqrt{d})} (4) &= \frac{\pi^8}{270 D_0^{5/2}} \sum_{k=1}^{D_0-1} \chi_{D_0} (k) \left( 2k^2 - \frac{k^4}{D_0^2} \right) . \end{split} \]
\end{lem}

\begin{proof} From the observation
\[ L(s,\chi_{D_0} ) = \sum_{k=1}^{D_0-1} \chi_{D_0} (k) \sum_{n=0}^{\infty} \frac{1}{(nD_0+k)^s} , \qquad \mathrm{Re} \, s >1 \]
and the fact that $\chi_{D_0}$ is a primitive Dirichlet character mod $D_0$, one easily deduces
\[ L(s,\chi_{D_0}) = \frac{1}{G(\chi_{D_0})} \sum_{k=1}^{D_0-1} \chi_{D_0} (k) \sum_{n=1}^{\infty} \frac{e^{2 \pi i n k/D_0}}{n^s}, \qquad \mathrm{Re} \, s >1, \]
where the Gauss sum $G(\chi_{D_0})=\sum_{k=1}^{D_0} \chi_{D_0} (k) e^{2 \pi i k/D_0}$ has value $G(\chi_{D_0}) = D_0^{1/2}$, see \cite[p.\ 262]{AP} and \cite[p.\ 49]{IK}.

Since $\chi_{D_0}(k)$ is even and periodic mod $D_0$, for real values of $s$ we have
\[ L(s,\chi_{D_0}) = \frac{1}{D_0^{1/2}} \sum_{k=1}^{D_0-1} \chi_{D_0} (k) \sum_{n=1}^{\infty} \frac{\cos(2 \pi n k/D_0)}{n^s}, \qquad s>1. \]
For even integer values of $s>1$, the trigonometric series in the previous formula is a Bernoulli polynomial. In particular,
\[ \begin{split} L(2,\chi_{D_0}) &= \frac{\pi^2}{D_0^{1/2}} \sum_{k=1}^{D_0-1} \chi_{D_0} (k) \left( \frac{k^2}{D_0^2} - \frac{k}{D_0} + \frac{1}{6} \right) , \\ L(4,\chi_{D_0}) &= - \frac{\pi^4}{3D_0^{1/2}} \sum_{k=1}^{D_0-1} \chi_{D_0} (k) \left( \frac{k^4}{D_0^4} - 2 \frac{k^3}{D_0^3} + \frac{k^2}{D_0^2} - \frac{1}{30} \right) . \end{split} \]
As $\chi_{D_0}$ is an even, primitive Dirichlet character, we have $\sum_{k=1}^{D_0-1} \chi_{D_0} (k) =0$ and $\sum_{k=1}^{D_0-1} \chi_{D_0} (k) k =0$. Further,
\[ \sum_{k=1}^{D_0-1} \chi_{D_0} (k) k^3 = \sum_{k=1}^{D_0-1} \chi_{D_0} (k) (D_0-k)^3 = \sum_{k=1}^{D_0-1} \chi_{D_0} (k) (3D_0 k^2 - k^3) , \]
consequently $\sum_{k=1}^{D_0-1} \chi_{D_0} (k) k^3 = (3D_0/2) \sum_{k=1}^{D_0-1} \chi_{D_0} (k) k^2$. Therefore,
\[ \begin{split} L(2,\chi_{D_0}) &= \frac{\pi^2}{D_0^{5/2}} \sum_{k=1}^{D_0-1} \chi_{D_0} (k) k^2 , \\ L(4,\chi_{D_0}) &= \frac{\pi^4}{3 D_0^{5/2}} \sum_{k=1}^{D_0-1} \chi_{D_0} (k) \left( 2k^2 - \frac{k^4}{D_0^2} \right) , \end{split} \]
and the claim follows from the factorization $\zeta_{\mathbb{Q}(\sqrt{d})}(s) = \zeta(s) L(s,\chi_{D_0})$.
\end{proof}

\subsubsection{The zeta function of a module}

Let $A=\mathbb{Z}\xi_1 + \mathbb{Z}\xi_2$ be a $\mathbb{Z}$-module of rank 2, and let $\{ \xi \in \mathbb{Q}(\sqrt{d}) : \xi A \subseteq A \}=\mathcal{O}_f$ with an integer $f \ge 1$. Thus $A$ has discriminant $N(A)^2 D=N(A)^2 f^2 D_0$. The zeta function of $A$ is defined as
\[ \zeta (A,s) = \sum_{\substack{\xi \in A /U_f^+ \\ \xi \textrm{ totally positive}}} \frac{N(A)^s}{N(\xi)^s}, \qquad \mathrm{Re} \, s>1. \]
Note that $\zeta(A,s)$ depends only on the equivalence class of $A$. Clearly, the zeta function of the conjugate module $\overline{A}=\mathbb{Z}\overline{\xi_2} + \mathbb{Z}\overline{\xi_1}$ is the same as that of $A$. The zeta function $\zeta(A,s)$ also has a meromorphic continuation to $\mathbb{C}$ with a simple pole at $s=1$.

The special values of $\zeta(A,k)$ at integers $k>1$ were found by Vlasenko and Zagier \cite{VZ}, and involve special functions such as polylogarithms. On the other hand, a functional equation due to Kramer \cite{KR} (see also \cite[p.\ 42]{VZ}) states that with an arbitrary $w \in \mathbb{Q}(\sqrt{d})$ such that $N(w)<0$, for all integers $k>1$,
\begin{equation}\label{zetafunctionalequation}
\zeta(A,k) + (-1)^k \zeta(wA,k) = \frac{2^{2k-1} \pi^{2k}}{D^{k-1/2} ((k-1)!)^2} \zeta (A,1-k) .
\end{equation}
The values of $\zeta(A,1-k)$ are rational, and were found by Shintani \cite{SH} and Zagier \cite{ZA2}. We refer to \cite{BG,GP,JL,ZA1} for further results on special values of zeta functions.

In the special case $f=1$, the equivalence classes of $\mathbb{Z}$-modules of rank 2 are in one-to-one correspondence with the narrow ideal classes in the ring of algebraic integers $\mathcal{O}$. Summing over all narrow ideal classes leads back to the Dedekind zeta function: $\sum_{A \textrm{ narrow ideal class}} \zeta(A,s) = \zeta_{\mathbb{Q}(\sqrt{d})}(s)$.

\subsection{Computing the constants $c_1(\alpha)$ and $c_2(\alpha)$}\label{computingsection}

We saw that Theorem \ref{maintheorem} holds with the constants $c_1 (\alpha) = c(\alpha, 2)/(4 \pi^4)$ and $c_2 (\alpha) = c(\alpha, 4)/(8 \pi^8)$, where
\[ c(\alpha, \theta) = \frac{D^{\theta/2}}{\log \varepsilon} \sum_{\substack{n \in \mathbb{Z} \\ n \neq 0}} \frac{R_Q(n)}{|n|^{\theta}} = \frac{D^{\theta/2}}{\log \varepsilon} \sum_{n=1}^{\infty} \frac{R_Q(n) + R_{-Q}(n)}{n^{\theta}} . \]
In particular, the functional equations \eqref{c1c2functionalequation} are special cases of \eqref{cfunctionalequation}.

It is an unfortunate fact that there is no general formula for the number of primary representations $R_Q (n)$. It is nevertheless possible to explicitly compute $c(\alpha, \theta)$ for integer values of $\theta>1$, in particular for $\theta=2$ and $\theta=4$, by expressing it in terms of zeta functions. We first prove Theorem \ref{c1c2theorem}, then consider some special cases in Sections \ref{allformssection} and \ref{genussection}. The case of a general quadratic irrational $\alpha$ is discussed in Section \ref{generalquadraticsection}.

\begin{proof}[Proof of Theorem \ref{c1c2theorem}] Let $\alpha$ be a quadratic irrational, let $ax^2+bx+c$ be its minimal polynomial with $a,b,c \in \mathbb{Z}$, $\mathrm{gcd}(a,b,c)=1$ and $a>0$, and let $Q(x,y)=ax^2 + bxy + cy^2$ be the corresponding binary quadratic form with discriminant $D=b^2-4ac$. Let $A_1=\mathbb{Z} + \mathbb{Z} \frac{b-\sqrt{D}}{2a}$ be the $\mathbb{Z}$-module corresponding to $Q$. In particular, $Q(x,y) = N(\xi)/N(A_1)$ with $\xi=x+y\frac{b-\sqrt{D}}{2a}$. As observed by Zagier \cite[p.\ 99]{ZA3}, for any positive integer $n$,
\[ |\{ (x,y) \in \mathbb{Z}^2 : Q(x,y)=n \} / \mathrm{Aut}(Q)| = |\{ \xi \in A_1 : \xi \textrm{ totally positive, } N(\xi)=n N(A_1) \}/U_f^+| , \]
since the equivalence classes are in one-to-one correpondance by the explicit description of $\mathrm{Aut}(Q)$ and $U_f^+$ in terms of $\varepsilon$. Thus by definition,
\[ \sum_{n=1}^{\infty} \frac{R_Q(n)}{n^s} = \zeta(A_1, s), \qquad \mathrm{Re} \, s>1. \]
We have either $\alpha=\frac{-b-\sqrt{D}}{2a}$ or $\alpha=\frac{-b+\sqrt{D}}{2a}$. Since the zeta function of a module and its conjugate are the same, we have $\zeta(A_1,s)=\zeta(A,s)$ with $A=\mathbb{Z} + \mathbb{Z} \alpha$. The $\mathbb{Z}$-module corresponding to $-Q(x,y)=-ax^2-bxy-cy^2$ is $A_2=\mathbb{Z}w+\mathbb{Z}w\frac{-b-\sqrt{D}}{-2a}$, where $w \in \mathbb{Q}(\sqrt{d})$ is arbitrary with $N(w)<0$. We similarly deduce
\[ \sum_{n=1}^{\infty} \frac{R_{-Q}(n)}{n^s} = \zeta(A_2, s), \qquad \mathrm{Re} \, s>1, \]
and here $\zeta(A_2,s)=\zeta(wA,s)$. Therefore,
\[ \sum_{n=1}^{\infty} \frac{R_Q(n) + R_{-Q}(n)}{n^s} = \zeta(A,s) + \zeta(wA,s), \qquad \mathrm{Re}\, s>1, \]
and setting $s=2$ resp.\ $s=4$ gives
\begin{equation}\label{c1c2zeta24}
c_1(\alpha) = \frac{D (\zeta(A,2) + \zeta(wA,2))}{4 \pi^4 \log \varepsilon} \quad \textrm{and} \quad c_2(\alpha) = \frac{D^2 (\zeta(A,4) + \zeta(wA,4))}{8 \pi^8 \log \varepsilon} .
\end{equation}
The functional equation \eqref{zetafunctionalequation} with $k=2$ resp.\ $k=4$ finally leads to the desired formulas
\[ c_1(\alpha) = \frac{2\zeta(A,-1)}{D^{1/2} \log \varepsilon} \quad \textrm{and} \quad c_2(\alpha) = \frac{4\zeta(A,-3)}{9 D^{3/2} \log \varepsilon} . \]
\end{proof}

\subsubsection{A special case with wide class number 1}\label{allformssection}

We first compute the constants $c_1(\alpha)$ and $c_2(\alpha)$ in a simple special case. For the sake of further simplicity, we also assume $f=1$, that is, the discriminant $D$ of the minimal polynomial of $\alpha$ is a fundamental discriminant.
\begin{lem}\label{classnumber12lemma} Assume $f=1$, and either $h(D)=1$, or $h(D)=2$ and $Q$ and $-Q$ are inequivalent for some binary quadratic form $Q$ of discriminant $D$. Then
\[ c_1 (\alpha) =  \frac{D}{2 \pi^4 h(D) \log \varepsilon} \zeta_{\mathbb{Q}(\sqrt{d})}(2) \quad \textrm{and} \quad c_2(\alpha) =  \frac{D^2}{4 \pi^8 h(D) \log \varepsilon} \zeta_{\mathbb{Q}(\sqrt{d})}(4) . \]
\end{lem}

\begin{proof} This immediately follows from \eqref{c1c2zeta24} and
\[ \zeta (A,s) + \zeta(wA,s) = \left\{ \begin{array}{ll} 2 \zeta_{\mathbb{Q}(\sqrt{d})}(s) & \textrm{if } h(D)=1, \\ \zeta_{\mathbb{Q}(\sqrt{d})}(s) & \textrm{if } h(D)=2 \textrm{ and } Q, -Q \textrm{ are inequivalent}.  \end{array} \right. \]
\end{proof}

We mention that the conditions of Lemma \ref{classnumber12lemma} are satisfied for a fundamental discriminant if and only if the class number of $\mathbb{Q}(\sqrt{d})$ is 1, that is, the wide ideal class group of the ring of algebraic integers is trivial. For the convenience of the reader, we include the list of all fundamental discriminants $1 \le D \le 100$ to which Lemma \ref{classnumber12lemma} applies: we have $h(D)=1$ for
\[ D=5, 8, 13, 17, 29, 37, 41, 53, 61, 73, 89, 97, \]
whereas $h(D)=2$ and $Q$, $-Q$ are inequivalent for some $Q$ for
\[ D=12, 21, 24, 28, 33, 44, 56, 57, 69, 76, 77, 88, 92, 93. \]
We refer to Buell \cite{BU} for extensive tables of class numbers.

Lemmas \ref{dedekindat24lemma} and \ref{classnumber12lemma} provide a convenient way of computing the constants $c_1(\alpha)$ and $c_2(\alpha)$. As an example, consider the golden ratio $\alpha=(1+\sqrt{5})/2$. The minimal polynomial $x^2-x-1$ has discriminant $D=5$, and the class number is $h(D)=1$. The smallest positive solution of Pell's equation $t^2-5u^2=4$ is $(t_0,u_0)=(3,1)$, hence $\varepsilon=\frac{3+\sqrt{5}}{2} = (\frac{1+\sqrt{5}}{2})^2$ (cf.\ Table \ref{epsilontable}). By Lemma \ref{dedekindat24lemma} (cf.\ Table \ref{zetatable}),
\[ \zeta_{\mathbb{Q}(\sqrt{5})}(2) = \frac{\pi^4}{6\cdot 5^{5/2}} \left( 1^2 - 2^2 - 3^2 + 4^2 \right) = \frac{2\sqrt{5}\pi^4}{375}, \]
and
\[ \begin{split} \zeta_{\mathbb{Q}(\sqrt{5})}(4) &= \frac{\pi^8}{270 \cdot 5^{5/2}} \left( \left( 2 \cdot 1^2 - \frac{1^4}{5^2} \right) - \left( 2 \cdot 2^2 - \frac{2^4}{5^2} \right) - \left( 2 \cdot 3^2 - \frac{3^4}{5^2} \right) + \left( 2 \cdot 4^2 - \frac{4^4}{5^2} \right) \right) \\ &= \frac{4 \sqrt{5}\pi^8}{84375} . \end{split} \]
Lemma \ref{classnumber12lemma} thus gives the explicit values
\[ \begin{split} c_1 \left( \frac{1+\sqrt{5}}{2} \right) &= \frac{5}{2 \pi^4 \log (\frac{1+\sqrt{5}}{2})^2} \cdot \frac{2 \sqrt{5}\pi^4}{375} = \frac{\sqrt{5}}{150 \log \frac{1+\sqrt{5}}{2}}, \\ c_2 \left( \frac{1+\sqrt{5}}{2} \right) &= \frac{5^2}{4 \pi^8 \log (\frac{1+\sqrt{5}}{2})^2} \cdot \frac{4 \sqrt{5} \pi^8}{84375} = \frac{\sqrt{5}}{6750 \log \frac{1+\sqrt{5}}{2}} . \end{split} \]
The values of $c_1(\alpha)$ and $c_2(\alpha)$ for $\alpha=\sqrt{2}, \sqrt{3}, \sqrt{6}, \sqrt{7}, \sqrt{11}$ and $\sqrt{14}$ in Table \ref{c1c2table} can be computed the same way from Lemmas \ref{dedekindat24lemma} and \ref{classnumber12lemma}.

\subsubsection{A special case with simple class group}\label{genussection}

Gauss introduced a binary operation called composition on the set of equivalence classes of binary quadratic forms with $\mathrm{gcd}(a,b,c)=1$ and discriminant $D$, which turns this set into a finite Abelian group of order $h(D)$, called the class group. The set of all squares forms a subgroup, and the cosets of this subgroup (that is, the elements of the quotient group) are called genera. We refer to Buell \cite{BU} for further details.

In \cite[Theorem 3]{MW}, an explicit formula was deduced for the sum of $\sum_{n=1}^{\infty} R_Q(n)/n^s$ over inequivalent binary quadratic forms $Q$ in a single genus. For the sake of simplicity, we only discuss fundamental discriminants, although the result actually applies to any discriminant.

Let $D>1$ be a fundamental discriminant. Each genus $G$ consists of $h(D)/2^{\omega(D)-1}$ equivalence classes of binary quadratic forms, where $\omega(D)$ is the number of distinct prime divisors of $D$. The elements of the set
\[ \{ -4,8,-8 \} \cup \{ p \, : \, p \equiv 1 \pmod{4} \textrm{ prime} \} \cup \{ -p \, : \, p \equiv 3 \pmod{4} \textrm{ prime} \} \]
are called prime discriminants. Any fundamental discriminant can be uniquely written as a product of pairwise coprime prime discriminants, and in fact any such product is a fundamental discriminant. Let $P(D)$ be the set of prime discriminants for which $D=\prod_{p^* \in P(D)}p^*$, and let $F(D)$ be the set of all products of distinct elements of $P(D)$ (including 1 as the empty product). In particular, $F(D)$ is the set of (positive or negative) fundamental discriminants that divide $D$. Fix a genus $G$, and let $Q_i$, $1 \le i \le h(D)/2^{\omega(D)-1}$ be representatives from the equivalence classes in $G$. For any $p^* \in P(D)$, define $\gamma_{p^*}(G)=(\frac{p^*}{k})$, where $k$ is a positive integer that is coprime with $p^*$ and is represented by $Q_i$ for some $1 \le i \le h(D)/2^{\omega(D)-1}$; it turns out that this does not depend on the choice of $k$. We extend it multiplicatively as $\gamma_{D'}(G) = \prod_{p^* \in P(D')} \gamma_{p^*} (G)$, $D' \in F(D)$ with the convention that $\gamma_1 (G)=1$ is the empty product. Then \cite[Theorem 3]{MW} states that
\begin{equation}\label{genusformula}
\sum_{i=1}^{h(D)/2^{\omega(D)-1}} \sum_{n=1}^{\infty} \frac{R_{Q_i}(n)}{n^s} = \frac{1}{2^{\omega(D)}} \sum_{D' \in F(D)} \gamma_{D'} (G) L(s,\chi_{D'}) L(s,\chi_{D/D'}), \quad \mathrm{Re} \, s>1.
\end{equation}
Note that $D/D'$ is also a fundamental discriminant for any $D' \in F(D)$. We use the convention $\chi_1(n)=(\frac{1}{n})=1$, so $L(s,\chi_1) = \zeta(s)$.

Formula \eqref{genusformula} involves Dirichlet characters $\chi_{D'}$ with negative discriminants $D'$. These are defined the same way in terms of the Kronecker symbol as $\chi_{D'}(n)=(\frac{D'}{n})$. As we will see, in our application the $L$-functions of these Dirichlet characters with negative discriminants will actually cancel, and Lemma \ref{genuslemma} only involves special values of Dedekind zeta functions of real quadratic fields.

In the special case when $h(D)=2^{\omega(D)-1}$, each genus contains only a single equivalence class, and formula \eqref{genusformula} immediately yields the values of $c(\alpha, \theta)$, $c_1(\alpha)$ and $c_2(\alpha)$. Note that this happens if and only if the class group is isomorphic to a direct product $\mathbb{Z}/(2\mathbb{Z}) \times \cdots \times \mathbb{Z}/(2\mathbb{Z})$, so that the subgroup of squares is trivial. We mention that \eqref{genusformula} also leads to the values of $c(\alpha,\theta)$, $c_1(\alpha)$ and $c_2(\alpha)$ when the equivalence classes of $Q$ and $-Q$ form a single genus of size 2. We omit the details, as in Section \ref{generalquadraticsection} we give a much more general method for computing these constants for arbitrary quadratic irrational $\alpha$.

In particular, \eqref{genusformula} leads to the following result for the principal form.
\begin{lem}\label{genuslemma} Let $d>1$ be a squarefree integer, and let
\[ Z_{d'}(s) = \left\{ \begin{array}{ll} \zeta(s) & \textrm{if } d'=1, \\ \zeta_{\mathbb{Q}(\sqrt{d'})} (s) / \zeta(s) & \textrm{if } d'>1 \textrm{ squarefree}. \end{array} \right. \]
\begin{enumerate}
\item[(i)] If $d \equiv 1 \pmod{4}$ and $h(d)=2^{\omega(d)-1}$, then
\[ \begin{split} c_1 \bigg( \frac{1+\sqrt{d}}{2} \bigg) &=\frac{d}{4 \pi^4 h(d) \log \varepsilon} \sum_{\substack{d' \mid d \\ d' \equiv 1 \pmod{4}}} Z_{d'}(2) Z_{d/d'}(2), \\  c_2 \bigg( \frac{1+\sqrt{d}}{2} \bigg) &=\frac{d^2}{8 \pi^8 h(d) \log \varepsilon} \sum_{\substack{d' \mid d \\ d' \equiv 1 \pmod{4}}} Z_{d'}(4) Z_{d/d'}(4) . \end{split} \]
\item[(ii)] If $d \equiv 3 \pmod{4}$ and $h(4d)=2^{\omega(d)}$, then
\[ \begin{split} c_1 (\sqrt{d}) &=\frac{d}{\pi^4 h(4d) \log \varepsilon} \sum_{d' \mid d} Z_{d'}(2) Z_{d/d'}(2), \\  c_2 (\sqrt{d}) &=\frac{2 d^2}{\pi^8 h(4d) \log \varepsilon} \sum_{d' \mid d} Z_{d'}(4) Z_{d/d'}(4) . \end{split} \]
\item[(iii)] If $d$ is even and $h(4d) = 2^{\omega(d)-1}$, then
\[ \begin{split} c_1 (\sqrt{d}) &=\frac{d}{\pi^4 h(4d) \log \varepsilon} \sum_{\substack{d' \mid d \\ d' \equiv 1, 5 \textrm{ or } d \pmod{8}}} Z_{d'}(2) Z_{d/d'}(2), \\  c_2 (\sqrt{d}) &=\frac{2 d^2}{\pi^8 h(4d) \log \varepsilon} \sum_{\substack{d' \mid d \\ d' \equiv 1, 5 \textrm{ or } d \pmod{8}}} Z_{d'}(4) Z_{d/d'}(4) . \end{split} \]
\end{enumerate}
\end{lem}

\begin{proof} We work with
\[ \alpha = \left\{ \begin{array}{ll} \frac{1+\sqrt{d}}{2} & \textrm{if } d \equiv 1 \pmod{4}, \\ \sqrt{d} & \textrm{if } d \equiv 2 \textrm{ or } 3 \pmod{4}, \end{array} \right. \]
and thus with the principal form
\[ Q(x,y) = \left\{ \begin{array}{ll} x^2 - xy - \frac{d-1}{4} y^2 & \textrm{if } d \equiv 1 \pmod{4}, \\ x^2 - d y^2 & \textrm{if } d \equiv 2 \textrm{ or 3} \pmod{4}, \end{array} \right. \]
whose discriminant
\[ D = \left\{ \begin{array}{ll} d & \textrm{if } d \equiv 1 \pmod{4}, \\ 4d & \textrm{if } d \equiv 2 \textrm{ or } 3 \pmod{4} \end{array} \right. \]
is a fundamental discriminant. By assumption, $h(D)=2^{\omega(D)-1}$, so each genus has size 1.

Assume $d \equiv 1 \pmod{4}$, and let us prove (i). The unique way to write the discriminant as a product of pairwise coprime prime discriminants is
\[ D=d=\prod_{\substack{p \mid d \\ p \equiv 1 \pmod{4}}} p \prod_{\substack{p \mid d \\ p \equiv 3 \pmod{4}}} (-p),  \]
where the second product has an even number of factors. The form $Q$ represents $k=Q(1,0)=1$, hence the genus $G_1$ consisting of the equivalence class of $Q$ satisfies $\gamma_{p^*}(G_1)=(\frac{p^*}{1}) = 1$, and by multiplicativity, $\gamma_{D'}(G_1)=1$ for all $D' \in F(D)$. Similarly, the form $-Q$ represents $k=-Q(1,1)=(d-1)/4$, so the genus $G_2$ consisting of the equivalence class of $-Q$ satisfies
\[ \gamma_{p^*}(G_2) = \left( \frac{p^*}{(d-1)/4} \right) = \left(\frac{p^*}{d-1} \right) = \left( \frac{p^*}{-1} \right)=\mathrm{sign}(p^*) ,\]
where we used the fact that the Dirichlet character $\chi_{p^*}(n)=(\frac{p^*}{n})$ is periodic mod $d$, completely multiplicative, and $\chi_{p^*}(4)=1$. By multiplicativity, $\gamma_{D'}(G_2) = \mathrm{sign}(D')$ for all $D' \in F(D)$. An application of formula \eqref{genusformula} thus gives
\[ \begin{split} \sum_{n=1}^{\infty} \frac{R_Q(n) + R_{-Q}(n)}{n^s} &= \frac{1}{2 h(d)} \sum_{D' \in F(D)} \left( 1+\mathrm{sign}(D') \right) L(s,\chi_{D'}) L(s,\chi_{D/D'}) \\ &= \frac{1}{h(d)} \sum_{\substack{D' \in F(D) \\ D'>0}}L(s,\chi_{D'}) L(s,\chi_{D/D'}) . \end{split} \]
The positive elements of $F(D)$ are exactly those divisors of $D=d$ that are divisible by an even number of primes $p \equiv 3 \pmod {4}$, hence $\{ D' \in F(D) : D'>0 \} = \{ d' : d' \mid d, \, d' \equiv 1 \pmod{4} \}$. In particular, for any $\theta>1$,
\[ c(\alpha, \theta) = \frac{d^{\theta/2}}{\log \varepsilon} \sum_{n=1}^{\infty} \frac{R_Q (n) + R_{-Q}(n)}{n^{\theta}} = \frac{d^{\theta/2}}{h(d)\log \varepsilon} \sum_{\substack{d' \mid d \\ d' \equiv 1 \pmod{4}}} L(\theta ,\chi_{d'}) L(\theta ,\chi_{d/d'}) . \]
Here $L(\theta, \chi_{d'}) = Z_{d'}(\theta)$, and the claim follows from setting $\theta=2$ resp.\ $\theta=4$.

Assume now $d \equiv 3 \pmod{4}$, and let us prove (ii). The unique way to write the discriminant as a product of pairwise coprime prime discriminants is
\[ D=4d=(-4) \prod_{\substack{p \mid d \\ p \equiv 1 \pmod{4}}} p \prod_{\substack{p \mid d \\ p \equiv 3 \pmod{4}}} (-p),  \]
where the second product has an odd number of factors. We similarly deduce
\[ \sum_{n=1}^{\infty} \frac{R_Q(n) + R_{-Q}(n)}{n^s} = \frac{1}{h(4d)} \sum_{\substack{D' \in F(D) \\ D'>0}}L(s,\chi_{D'}) L(s,\chi_{D/D'}) . \]
The positive elements of $F(D)$ are exactly those divisors of $d$ that are divisible by an even number of primes $p \equiv 3 \pmod{4}$, and 4 times those divisors of $d$ that are divisible by an odd number of primes $p \equiv 3 \pmod{4}$. In particular,
\[ \{ D' \in F(D) : D'>0 \} = \{ d' : d' \mid d, \, d' \equiv 1 \pmod{4} \} \cup \{ 4d' : d' \mid d, \, d' \equiv 3 \pmod{4} \}, \]
hence for any $\theta>1$,
\[ \begin{split} c(\alpha, \theta) &= \frac{(4d)^{\theta/2}}{h(4d) \log \varepsilon} \Bigg( \sum_{\substack{d' \mid d \\ d' \equiv 1 \pmod{4}}} L(\theta, \chi_{d'}) L (\theta, \chi_{4d/d'}) +\sum_{\substack{d' \mid d \\ d' \equiv 3 \pmod{4}}} L(\theta, \chi_{4d'}) L (\theta, \chi_{d/d'}) \Bigg) \\ &= \frac{(4d)^{\theta/2}}{h(4d) \log \varepsilon} \Bigg( \sum_{\substack{d' \mid d \\ d' \equiv 1 \pmod{4}}} Z_{d'}(\theta) Z_{d/d'}(\theta) +\sum_{\substack{d' \mid d \\ d' \equiv 3 \pmod{4}}} Z_{d'}(\theta) Z_{d/d'}(\theta) \Bigg) \\ &=\frac{(4d)^{\theta/2}}{h(4d) \log \varepsilon} \sum_{d' \mid d} Z_{d'}(\theta) Z_{d/d'}(\theta) . \end{split} \]
The claim follows from setting $\theta=2$ resp.\ $\theta=4$.

Finally, assume $d$ is even, and let us prove (iii). The unique way to write the discriminant as a product of pairwise coprime prime discriminants is
\[ D=4d=(w8) \prod_{\substack{p \mid \frac{d}{2} \\ p \equiv 1 \pmod{4}}} p \prod_{\substack{p \mid \frac{d}{2} \\ p \equiv 3 \pmod{4}}} (-p) \quad \textrm{with} \quad w=\left\{ \begin{array}{ll} +1 & \textrm{if } \frac{d}{2} \equiv 1 \pmod{4}, \\ -1 & \textrm{if } \frac{d}{2} \equiv 3 \pmod{4} . \end{array} \right.  \]
We similarly deduce
\[ \sum_{n=1}^{\infty} \frac{R_Q(n) + R_{-Q}(n)}{n^s} = \frac{1}{h(4d)} \sum_{\substack{D' \in F(D) \\ D'>0}}L(s,\chi_{D'}) L(s,\chi_{D/D'}) . \]
The positive elements of $F(D)$ are exactly those divisors of $d/2$ that are divisible by an even number of primes $p \equiv 3 \pmod{4}$, and 8 times those divisors of $d/2$ that are divisible by an even (resp.\ odd) number of primes $p \equiv 3 \pmod{4}$ if $w=1$ (resp.\ $w=-1$). In particular,
\[ \{ D' \in F(D) : D'>0 \} = \left\{ d' : d' \mid \frac{d}{2}, \, d' \equiv 1 \pmod {4} \right\} \cup \left\{ 8d' : d' \mid \frac{d}{2}, \, d' \equiv \frac{d}{2} \pmod{4} \right\} , \]
hence for any $\theta>1$,
\[ \begin{split} c(\alpha, \theta) &= \frac{(4d)^{\theta/2}}{h(4d) \log \varepsilon} \Bigg( \sum_{\substack{d' \mid \frac{d}{2} \\ d' \equiv 1 \pmod{4}}} L(\theta, \chi_{d'}) L (\theta, \chi_{4d/d'}) +\sum_{\substack{d' \mid \frac{d}{2} \\ d' \equiv \frac{d}{2} \pmod{4}}} L(\theta, \chi_{8d'}) L (\theta, \chi_{d/(2d')}) \Bigg) \\ &= \frac{(4d)^{\theta/2}}{h(4d) \log \varepsilon} \Bigg( \sum_{\substack{d' \mid \frac{d}{2} \\ d' \equiv 1 \pmod{4}}} Z_{d'}(\theta) Z_{d/d'}(\theta) +\sum_{\substack{d' \mid \frac{d}{2} \\ d' \equiv \frac{d}{2} \pmod{4}}} Z_{2d'}(\theta) Z_{d/(2d')}(\theta) \Bigg) \\ &= \frac{(4d)^{\theta/2}}{h(4d) \log \varepsilon} \Bigg( \sum_{\substack{d' \mid d \\ d' \equiv 1 \pmod{4}}} Z_{d'}(\theta) Z_{d/d'}(\theta) +\sum_{\substack{d' \mid d \\ d' \equiv d \pmod{8}}} Z_{d'}(\theta) Z_{d/d'}(\theta) \Bigg) . \end{split} \]
The claim follows from setting $\theta=2$ resp.\ $\theta=4$.
\end{proof}

As an example, consider $d=30$. The discriminant $D=120$ has class number $h(D)=4$. The smallest positive solution of Pell's equation $t^2-120u^2=4$ is $(t_0,u_0)=(22,2)$, hence $\varepsilon=11+2\sqrt{30}$. The set of divisors of $d$ is $\{ 1,2,3,5,6,10,15,30 \}$, but only $\{ 1,5,6,30 \}$ of them are congruent to 1, 5 or $d$ modulo 8. Lemma \ref{genuslemma} thus gives
\[ \begin{split} c_1 (\sqrt{30}) &=\frac{30}{\pi^4 4 \log (11+2\sqrt{30})} 2 \left( \zeta_{\mathbb{Q}(\sqrt{30})}(2) + \frac{\zeta_{\mathbb{Q}(\sqrt{5})}(2)}{\zeta(2)} \cdot \frac{\zeta_{\mathbb{Q}(\sqrt{6})}(2)}{\zeta(2)} \right) , \\  c_2 (\sqrt{30}) &=\frac{2\cdot 30^2}{\pi^8 4 \log (11+2\sqrt{30})} 2 \left( \zeta_{\mathbb{Q}(\sqrt{30})}(4) + \frac{\zeta_{\mathbb{Q}(\sqrt{5})}(4)}{\zeta(4)} \cdot \frac{\zeta_{\mathbb{Q}(\sqrt{6})}(4)}{\zeta(4)} \right) . \end{split} \]
An application of Lemma \ref{dedekindat24lemma} leads to (cf.\ Table \ref{zetatable})
\[ \zeta_{\mathbb{Q}(\sqrt{5})}(2) = \frac{2 \sqrt{5} \pi^4}{375}, \qquad \zeta_{\mathbb{Q}(\sqrt{6})}(2) = \frac{\sqrt{6}\pi^4}{144}, \qquad \zeta_{\mathbb{Q}(\sqrt{30})}(2) = \frac{17\sqrt{30}\pi^4}{5400} \]
and
\[ \zeta_{\mathbb{Q}(\sqrt{5})}(4) = \frac{4 \sqrt{5} \pi^8}{84375}, \qquad \zeta_{\mathbb{Q}(\sqrt{6})}(4) = \frac{29 \sqrt{6}\pi^8}{622080}, \qquad \zeta_{\mathbb{Q}(\sqrt{30})}(4) = \frac{36451\sqrt{30}\pi^8}{1749600000} , \]
and we finally deduce the explicit values
\[ c_1(\sqrt{30}) = \frac{121 \sqrt{30}}{1800 \log (11+2\sqrt{30})} \quad \textrm{and} \quad c_2(\sqrt{30}) = \frac{28224541\sqrt{30}}{1944000 \log (11+2\sqrt{30})} . \]

\subsubsection{General quadratic irrationals}\label{generalquadraticsection}

In some special cases when the class group is of a suitable form, Lemmas \ref{classnumber12lemma} and \ref{genuslemma} give the explicit value of the constants $c_1(\alpha)$ and $c_2(\alpha)$ in terms of special values of the Dedekind zeta function. The main advantage is that these special values are easily computable by Lemma \ref{dedekindat24lemma}.

We now consider the case of a general quadratic irrational $\alpha$, without any assumption on the class group. Theorem \ref{c1c2theorem} expresses $c_1(\alpha)$ and $c_2(\alpha)$ in terms of special values of zeta functions at $s=-1$ and $s=-3$. Following the general method of Zagier \cite{ZA2} for finding the values of zeta functions at negative integers leads to the explicit values of $c_1(\alpha)$ and $c_2(\alpha)$ for an arbitrary $\alpha$.

We now recall the method of Zagier. With the notation in \ref{quadraticfieldssection}, let $A=\mathbb{Z}\xi_1 + \mathbb{Z}\xi_2$ be a $\mathbb{Z}$-module of rank 2 with an oriented basis, that is, $\xi_1 \overline{\xi_2} - \overline{\xi_1} \xi_2>0$, such that the order $\{ \xi \in \mathbb{Q}(\sqrt{d}) : \xi A \subseteq A \}$ has discriminant $D=f^2 D_0$. The quadratic irrational $\xi_1/\xi_2$ has a unique backward continued fraction expansion
\[ \frac{\xi_1}{\xi_2} = \llbracket b_0, b_1, b_2, \ldots \rrbracket = b_0 - \cfrac{1}{b_1-\cfrac{1}{b_2-\cdots}} \]
with an integer $b_0$ and integers $b_i \ge 2$, $i \ge 1$, which is eventually periodic with some minimal period length $r \ge 1$. In particular, $b_{i+r}=b_i$ for all $i \ge i_0$ with some $i_0$. Consider the purely periodic backward continued fractions
\[ w_j = \llbracket b_{i_0+j}, b_{i_0+j+1}, \ldots \rrbracket, \qquad 0 \le j \le r-1 . \]
It turns out that the set of all $\mathbb{Z}$-modules $\mathbb{Z}w+\mathbb{Z}$ that are equivalent to $A$ and are reduced in the sense that $w>1>\overline{w}>0$, is exactly $\mathbb{Z}w_j + \mathbb{Z}$, $0 \le j \le r-1$. Let
\[ Q_j (x,y) = \frac{\sqrt{D}}{w_j - \overline{w_j}} (x w_j +y)(x \overline{w_j}+y), \qquad 0 \le j \le r-1 \]
be the corresponding binary quadratic forms. Then for any integer $k \ge 1$,
\begin{equation}\label{zagierformula}
\zeta (A,-k)= \sum_{j=0}^{r-1} \sum_{i=0}^{2k} d_{i,k}^{(j)} \left( \frac{B_{2k+2}}{2k+2} \cdot \frac{b_{i_0+j}^{2k-i+1}}{2k-i+1} - \frac{B_{i+1}}{i+1} \cdot \frac{B_{2k-i+1}}{2k-i+1} \right) ,
\end{equation}
where $B_i$ is the sequence of Bernoulli numbers, and $d_{i,k}^{(j)}$ are the coefficients in
\[ Q_j (x,y)^k = \sum_{i=0}^{2k} (-1)^i d_{i,k}^{(j)} x^i y^{2k-i}, \qquad 0 \le j \le r-1 . \]

As a final example, let us choose a generic quadratic irrational, say, $\alpha = \frac{19+3\sqrt{69}}{26}$, and let us find the values of $c_1(\alpha)$ and $c_2(\alpha)$ using the method of Zagier. The minimal polynomial $13x^2 - 19x -5$ has discriminant $D=621$, which we write in the form $D=f^2 D_0$ with $f=3$ and the fundamental discriminant $D_0=69$. In particular, $\alpha \in \mathbb{Q}(\sqrt{d})$ with $d=69$. The smallest positive solution of Pell's equation $t^2-621u^2=4$ is $(t_0,u_0)=(25,1)$, hence $\varepsilon = (25+3\sqrt{69})/2$. We mention that the class number is $h(D)=6$.

We write the $\mathbb{Z}$-module $A=\mathbb{Z}+\mathbb{Z}\alpha$ in the form $A=\mathbb{Z}\alpha + \mathbb{Z}$ with the oriented basis $\alpha, 1$. The backward continued fraction expansion of $\alpha/1$ turns out to be $\alpha = \llbracket 2, \overline{4,2,2,2,3} \rrbracket$, where the overline denotes the period. The purely periodic backward continued fraction expansions with the corresponding binary quadratic forms are
\begin{align*}
w_0 &= \llbracket \overline{4,2,2,2,3} \rrbracket = \frac{11+\sqrt{69}}{6}, & Q_0(x,y) &= 13x^2 + 33xy + 9y^2, \\
w_1 &= \llbracket \overline{2,2,2,3,4} \rrbracket = \frac{39+3\sqrt{69}}{50}, & Q_1(x,y) &= 9x^2 + 39xy + 25y^2, \\
w_2 &= \llbracket \overline{2,2,3,4,2} \rrbracket = \frac{61+3\sqrt{69}}{62}, & Q_2(x,y) &= 25x^2 + 61xy + 31y^2, \\
w_3 &= \llbracket \overline{2,3,4,2,2} \rrbracket = \frac{21+\sqrt{69}}{18}, & Q_3(x,y) &= 31x^2 + 63xy + 27y^2, \\
w_4 &= \llbracket \overline{3,4,2,2,2} \rrbracket = \frac{45+3\sqrt{69}}{26}, & Q_4(x,y) &= 27x^2 + 45xy + 13y^2.
\end{align*}
The first 8 Bernoulli numbers are
\[ B_1 =-\frac{1}{2}, \qquad B_2 = \frac{1}{6}, \qquad B_3 = B_5 = B_7 =0, \qquad B_4 = B_8 = - \frac{1}{30} , \qquad B_6 = \frac{1}{42} . \]
Setting $k=1$ in \eqref{zagierformula}, the terms $j=0,1,2,3,4$ are
\[ \begin{split} 9 \left( - \frac{1}{120} \cdot \frac{4^3}{3} \right) - 33 \left( - \frac{1}{120} \cdot \frac{4^2}{2} - \frac{1}{144} \right) + 13 \left( - \frac{1}{120} \cdot \frac{4^1}{1} \right) &= \frac{19}{48}, \\ 25 \left( - \frac{1}{120} \cdot \frac{2^3}{3} \right) - 39 \left( - \frac{1}{120} \cdot \frac{2^2}{2} - \frac{1}{144} \right) + 9 \left( - \frac{1}{120} \cdot \frac{2^1}{1} \right) &=\frac{31}{144}, \\ 31 \left( - \frac{1}{120} \cdot \frac{2^3}{3} \right) - 61 \left( - \frac{1}{120} \cdot \frac{2^2}{2} - \frac{1}{144} \right) + 25 \left( - \frac{1}{120} \cdot \frac{2^1}{1} \right) &= \frac{241}{720}, \\ 27 \left( - \frac{1}{120} \cdot \frac{2^3}{3} \right) - 63 \left( - \frac{1}{120} \cdot \frac{2^2}{2} - \frac{1}{144} \right) + 31 \left( - \frac{1}{120} \cdot \frac{2^1}{1} \right) &= \frac{89}{240}, \\ 13 \left( - \frac{1}{120} \cdot \frac{3^3}{3} \right) - 45 \left( - \frac{1}{120} \cdot \frac{3^2}{2} - \frac{1}{144} \right) + 27 \left( - \frac{1}{120} \cdot \frac{3^1}{1} \right) &= \frac{7}{20}, \end{split} \]
respectively, hence
\[ \zeta (A,-1) = \frac{19}{48} + \frac{31}{144} + \frac{241}{720} + \frac{89}{240} + \frac{7}{20} = \frac{5}{3} . \]
Next, set $k=3$ in \eqref{zagierformula}. Using the coefficients in the expansions of $Q_j(x,y)^3$, $j=0,1,2,3,4$ leads to the $j=0,1,2,3,4$ terms in \eqref{zagierformula}, and we obtain
\[ \begin{split} \zeta (A,-3) &= \frac{305161}{11200} + \frac{200087}{11200} + \frac{3846131}{33600} +\frac{4287701}{33600} + \frac{63901}{1400} = \frac{1997}{6} . \end{split} \]
Theorem \ref{c1c2theorem} thus gives the explicit values
\[ \begin{split} c_1 \left( \frac{19+3\sqrt{69}}{26} \right) &= \frac{2 \cdot \frac{5}{3}}{(621)^{1/2} \log \frac{25+3 \sqrt{69}}{2}} = \frac{10 \sqrt{69}}{621 \log \frac{25+3\sqrt{69}}{2}}, \\ c_2 \left( \frac{19+3\sqrt{69}}{26} \right) &= \frac{4 \cdot \frac{1997}{6}}{9(621)^{3/2} \log \frac{25+3\sqrt{69}}{2}} = \frac{3994 \sqrt{69}}{3470769 \log \frac{25+3\sqrt{69}}{2}} . \end{split} \]

\end{document}